\documentclass{article}
\usepackage{amssymb,amsfonts,amsmath,amsthm,amsopn,amstext,amscd,latexsym,xy,mathdots,stmaryrd}
\usepackage{hyperref}
\theoremstyle{plain}
\input xy
\xyoption{all}
\newtheorem{theorem}{Theorem}[section]
\newtheorem{lemma}[theorem]{Lemma}
\newtheorem{proposition}[theorem]{Proposition}
\newtheorem{prop}[theorem]{Proposition}

\newtheorem{cor}[theorem]{Corollary}
\newtheorem{definition}[theorem]{Definition}

\theoremstyle{remark}

\newtheorem*{claim}{Claim}
\numberwithin{equation}{section}
\numberwithin{paragraph}{section}

\DeclareMathOperator{\Hom}{Hom}

\DeclareMathOperator{\Ad}{Ad}

\DeclareMathOperator{\rec}{rec}

\DeclareMathOperator{\Gal}{Gal}

\DeclareMathOperator{\ad}{Ad}

\DeclareMathOperator{\Res}{Res}

\DeclareMathOperator{\Ind}{Ind}
\DeclareMathOperator{\val}{val}

\DeclareMathOperator{\tr}{tr}
\DeclareMathOperator{\St}{St}
\DeclareMathOperator{\Supp}{Supp}
\DeclareMathOperator{\diag}{diag}
\DeclareMathOperator{\Art}{Art}
\DeclareMathOperator{\End}{End}

\DeclareMathOperator{\Frob}{Frob}

\DeclareMathOperator{\Spec}{Spec}

\newcommand{\cD}{{\mathcal D}}

\newcommand{\cG}{{\mathcal G}}

\newcommand{\cJ}{{\mathcal J}}

\newcommand{\cL}{{\mathcal L}}

\newcommand{\cO}{{\mathcal O}}

\newcommand{\cR}{{\mathcal R}}
\newcommand{\cS}{{\mathcal S}}

\newcommand{\ffrm}{{\mathfrak m}}

\newcommand{\bbA}{{\mathbb A}}

\newcommand{\bbC}{{\mathbb C}}

\newcommand{\bbF}{{\mathbb F}}

\newcommand{\bbQ}{{\mathbb Q}}
\newcommand{\bbR}{{\mathbb R}}

\newcommand{\bbT}{{\mathbb T}}

\newcommand{\bbZ}{{\mathbb Z}}

\newcommand{\GL}{\mathrm{GL}}

\newcommand{\PGL}{\mathrm{PGL}}

\newcommand{\SL}{\mathrm{SL}}

\newcommand{\PSL}{\mathrm{PSL}}

\newcommand{\wv}{{\widetilde{v}}}

\newcommand{\barqp}{{\overline{\bbQ}_p}}

\newcommand{\barfp}{{\overline{\bbF}_p}}

\newcommand{\CNL}{{\mathrm{CNL}}}
\newcommand{\Sets}{{\mathrm{Sets}}}

\newcommand{\T}{\mathbb{T}}
\newcommand{\A}{\mathbb{A}}
\newcommand{\Q}{\mathbb{Q}}

\newcommand{\C}{\mathbb{C}}

\newcommand{\Z}{\mathbb{Z}}

\newcommand{\aFp}{\overline{\mathbb{F}}_{p}}

\newcommand{\aZp}{\overline{\mathbb{Z}}_{p}}
\newcommand{\aQp}{\overline{\mathbb{Q}}_{p}}

\newcommand{\Oc}{\mathcal O}

\newcommand{\univ}{{\text{univ}}}

\def\rhobar{ {\overline{\rho}} }

\newcommand{\ra}{\rightarrow}

\newcommand{\F}{{\mathbb F}}

\title{Automorphy of some residually  $S_5$ Galois representations}

\author{Chandrashekhar B.  Khare\footnote{\textsc{Department of Mathematics, UCLA, Los Angeles, USA.} \textit{Email address}: \texttt{shekhar@math.ucla.edu}}  \ \ and  \ \ Jack A. Thorne\footnote{\textsc{Department of Pure Mathematics and Mathematical Statistics, Wilberforce Road, Cambridge, United Kingdom.} \textit{Email address:} \texttt{thorne@dpmms.cam.ac.uk}}} 

\begin{document}
\maketitle

\begin{abstract}
Let $F$ be a totally real field and $p$ an odd prime. 
We prove an automorphy lifting  theorem for  geometric representations $\rho:G_F \ra \GL_2(\aQp)$ 
which lift irreducible residual representations $\rhobar$ that arise from Hilbert modular forms. The new result  is that we allow  the case $p=5$, $\rhobar$ has projective image $S_5\cong \PGL_2(\F_5)$ and the fixed field of the kernel of the  projective representation contains $\zeta_5$.  The usual Taylor--Wiles method does not work in this case as there are elements of dual Selmer that cannot be killed by allowing ramification at Taylor--Wiles primes.  These elements arise from our hypothesis and the non-vanishing of $H^1(\PGL_2(\F_5),\Ad(1))$ where $\Ad(1)$ is the adjoint of the natural representation of $\GL_2(\F_5)$ twisted by the quadratic character of $\PGL_2(\F_5)$.  
\end{abstract}

\tableofcontents

\section{Introduction}

Let $F$ be a totally real number field, let $p$ be a prime, and let $\rho : G_F \to \GL_2(\overline{\bbQ}_p)$ be a geometric Galois representation. Conjectures of Fontaine--Mazur and Clozel predict that $\rho$ should be \emph{automorphic}, i.e.\ associated to a cuspidal automorphic representation of $\GL_2(\bbA_F)$. One can often prove this when one has access to 
an automorphy lifting theorem. Automorphy liftings theorems show, broadly speaking,  that  a geometric representation $\rho$ is automorphic, under the hypothesis that there is a mod $p$ congruence between $\rho$ and another geometric Galois representation $\rho'$ which is already known to be automorphic. In other words, they imply the automorphy of $\rho$, conditional on the \emph{residual automorphy} of the residual representation $\overline{\rho} : G_F \to \GL_2(\overline{\bbF}_p)$. In this paper, we prove a new automorphy lifting result for 2-dimensional Galois representations by using a variation of the techniques of an earlier paper of the second author \cite{ThorneA}. Our main result is the following.
\begin{theorem}\label{main}
 Let $F$ be a totally real field, $p>2$ a prime, and $\aQp$ an algebraic closure of $\Q_p$. Let $\rho:G_F \rightarrow \GL_2(\aQp)$ be a continuous representation satisfying the following conditions.
\begin{enumerate} \item The representation $\rho$ is unramified almost everywhere.
 \item For each place $v|p$ of $F$, $\rho|_{G_{F_v}}$ is de Rham. For each embedding $\tau:F \hookrightarrow \aQp$, we have  ${\rm HT}_\tau(\rho)=\{0,1\}$.
 \item The residual representation $\rhobar:G_F \rightarrow \GL_2(\aFp)$ is irreducible when restricted to $G_{F(\zeta_p)}$ and arises from a cuspidal automorphic representation $\pi$ of $\GL_2(\A_F)$ of weight 2.
\end{enumerate}
 Then $\rho$ is automorphic: there exists a cuspidal automorphic representation $\pi$ of $\GL_2(\A_F)$ of weight 2 and an isomorphism $\iota:\aQp \cong \C$ such that $\rho \simeq r_\iota(\pi)$.
 \end{theorem}
We invite the reader to compare this statement with \cite[Theorem 3.5.5]{Kisin}. The new case as compared to \emph{loc. cit.} is when $p=5$, the  image of the projective  residual representation of $\rho$ (i.e. $\rhobar$ modulo the center)  is isomorphic to $\PGL_2(\F_5)$ and the fixed field of its kernel contains $\zeta_5$ (which implies that $[F(\zeta_5):F]=2$).  We call such $\rhobar$ {\it exceptional}.  It is easy to check that $\rhobar|_{G_{F'}}$ remains exceptional for any totally real extension  $F'/F$. A construction of Mestre gives  examples of exceptional $\rhobar$ (see \S \ref{examples} below).

We now explain the difficulty that this exceptional case presents. The most important tool for proving automorphy lifting theorems is the Taylor--Wiles argument. The original patching argument  of \cite{TW} proves an  $R=\T$ theorem, where $R$ is a certain deformation ring parametrizing deformations of the residual representation $\rhobar$ with local conditions (which are  in particular  almost everywhere the condition of being unramified), and $\T$ is a Hecke ring acting typically on a space of new forms $S_2(\Gamma_0(N),\Z)_{\mathfrak m}$ of suitable level $N$,  where $\mathfrak m$ is a maximal ideal of $\T$ associated to $\rhobar$.  Let $S$ be a finite set of primes that contains the prime divisors of $N$ and $p$. The  mod $p$ cotangent space of  $R$ is a certain Selmer group $H^1_{\mathcal L}(G_S,\Ad^0(\rhobar))$, with $\mathcal  L =(\mathcal L_v)_{v \in S}$ consisting of subspaces  $\mathcal L_v$ of $H^1(G_v,\Ad^0(\rhobar))$. 

One of the first steps in the Taylor--Wiles method is to choose auxiliary sets of primes $Q$ disjoint from $S$ and consider deformation rings
$R_Q$ with ramification allowed at $S \cup Q$, which surject onto  $R$,  such that $H^1_{\mathcal L_Q^\perp}(G_{S \cup Q},\Ad^0(\rhobar)(1))=0$.   (Here we just mention that  the Selmer condition $\cL_Q$ is a collection of subspaces of $H^1(G_v,\Ad^0(\rhobar))$ for each $v \in S \cup Q$, which at $v \in Q$ is all of $H^1(G_v,\Ad^0(\rhobar))$ and $\cL_Q^\perp =(\cL_v^\perp)_{v \in S \cup Q}$ consist of the orthogonal complements  of these subspaces under the local Tate duality pairing.) We call this  killing dual Selmer. 
By Wiles' Euler characteristic formula, this ensures that $H^1_{\mathcal L_Q}(G_{S \cup Q},\Ad^0(\rhobar))=H^1_{\mathcal L}(G_S,\Ad^0(\rhobar))$.

 Using such carefully chosen (Taylor--Wiles)  primes $Q$ one  asymptotically approximates the statement that there is a unique lift of $\rhobar$
with given inertial behavior (which corresponds to abelian/principal series ramification) at auxiliary primes.  As there is always a lift with given inertial behavior (which at almost all places is the unramified condition) that  arises from a newform, the Taylor--Wiles patching argument deduces the modularity of the given lift $\rho$.  The actual method is more intricate as the uniqueness statement is only asymptotic  and approximate: thus the Taylor--Wiles patching is necessitated. (The patching exploits the lower bound on the growth of Hecke rings at finite levels given by producing congruences, that  ``matches''  better and better  the upper bound on the size of  deformation rings at finite levels.)

This process of killing the dual Selmer group can only be carried out under some auxiliary hypotheses on the residual representation $\overline{\rho}$. The most general condition is that of adequacy, as defined in \cite[\S 2]{ThorneB}. Here, this boils down to the assertion that $\overline{\rho}|_{G_{F(\zeta_p)}}$ is irreducible and that $H^1(H, \ad^0) = 0$, where $H \subset \PGL_2(\overline{\bbF}_p)$ is the projective image of $\overline{\rho}|_{G_{F(\zeta_p)}}$ and $\ad^0$ is the natural adjoint representation of this group. 

In a previous work \cite{ThorneA}, we proved automorphy lifting theorems in some cases when the representation $\overline{\rho}|_{G_{F(\zeta_p)}}$ is reducible. The idea is to consider for each $N \geq 1$ modified local deformation conditions, which are satisfied by the representation $\rho \text{ mod }p^N$ (but not by $\rho$ itself). If these local conditions are chosen correctly, this allows one to convert troublesome dual Selmer classes (which cannot be killed by Taylor--Wiles primes) to classes which can indeed be killed. The usual Taylor--Wiles argument then allows one to deduce the automorphy of $\rho \text{ mod }p^N$. An argument of level-lowering mod $p^N$ and passage to the limit then implies the automorphy of $\rho$ itself. 

This argument was inspired by earlier works of the first named author \cite{K}, \cite{K1}, which use techniques of Ramakrishna to reduce to a situation where $\rho \text{ mod }p^N$ lives in a universal deformation ring which has a trivial tangent space, so a unique $\overline{\bbQ}_p$-point. This allows one to conclude automorphy  of $\rho \text{ mod }p^N$ by a rigidity argument. However, these techniques of Ramakrishna are not universally applicable, necessitating the use of Taylor--Wiles primes to get the most general result.

In this paper, we use the techniques of \cite{ThorneA} in a similar way. In contrast to that paper, we consider residual characteristic 5 representations $\overline{\rho}$ which do remain irreducible on restriction to the group $G_{F(\zeta_5)}$, but which fail adequacy for a different reason: the group $H^1(H, \ad^0) = H^1(\PSL_2(\bbF_5), \ad^0)$ is non-zero. Under the  assumption that the projective image of $\rhobar$  is $\PGL_2(\F_5)$ and its fixed field contains $F(\zeta_5)$, i.e. $\rhobar$ is exceptional,  this leads to classes in the dual Selmer group which arise by inflation from the image of $\overline{\rho}$. We call these classes  {\it Lie classes} or {\it Lie elements}. For exceptional $\rhobar$,   Lie elements  of dual Selmer   cannot be killed by Taylor--Wiles primes. Indeed, a Frobenius element at a Taylor--Wiles prime has order prime to $p$, and every element of the group $H^1(\PSL_2(\bbF_5), \ad^0)$ is killed on restriction to a subgroup of order prime to $p$. 

We show however that these Lie elements remain non-trivial on restriction to a decomposition group at a place $v$ where $\rhobar$ is unramified, the norm $q_v$ is 1 modulo $p$,    and  the image under $\rhobar$ of Frobenius at $v$ has order divisible by $p$. (Such places were used in \cite{Tay} to kill Lie elements of Selmer groups when the projective image of $\rhobar$ is $\PSL_2(\F_5)=A_5$. We use them to kill Lie elements of dual Selmer  for exceptional $\rhobar$.) By allowing ramification at such primes, we can again convert these troublesome elements into ones which are easily killed. The rest of the argument using automorphy mod $p^N$ is more or less the same. 

An important role is played by Lemma \ref{Z5}, whose proof follows from arguments in \cite{Boe},  which implies that whenever $\overline{\rho}$ has projective image $\PGL_2(\bbF_5)$, the projective image of $\rho$ contains a conjugate of $\PGL_2(\bbZ_5)$. This result, together with the Cebotarev density theorem, is used to find primes at which allowing ramification will have the desired effect on the dual Selmer group (cf. Proposition \ref{auxiliary}). This result is also essential to show that we can choose these primes $v$ to satisfy $q_v \equiv 1\text{ mod }p$ but  $q_v \not \equiv 1 \text{ mod }p^2$; this extra condition is needed to control the error terms in the process of level-lowering modulo $p^N$ (cf. Theorem \ref{thm_level_lowering_mod_p_to_the_N}).

It is natural to ask if our results can be generalized to prove automorphy lifting theorems in non-adequate situations in higher dimensions (i.e. when $n > 2$). The answer is certainly yes, and in fact we prove in an inadequate situation a rather restrictive automorphy lifting theorem for $\GL_4$ in \S \ref{ordinary} on the way to constructing ordinary lifts of 2-dimensional representations. We have not attempted to understand what is the most general theorem that can be proved this way, outside of $\GL_2$.

We now describe the organization of this paper. In \S \ref{sec_normalizations}, we recall our notation and normalizations. As far as possible, we have tried to use the same notation as the earlier paper \cite{ThorneA}. We begin the main body of the paper in \S \ref{examples} by showing, following Mestre, that exceptional residual representations $\overline{\rho}$ are plentiful over any totally real field $F$ satisfying the necessary condition $\sqrt{5} \in F$. In \S \ref{group_theory}, we state some group theoretical results, and in particular prove the result of B\"ockle mentioned above. In \S \ref{sec_shimura_curves_and_hida_varieties}, we define certain Shimura curves and their 0-dimensional analogues, and study the relation between their cohomology groups. For the most part we simply recall results proved in \cite[\S 4]{ThorneA}, but sometimes we must work harder; for example, when we wish to understand level lowering/raising at a prime $v$ where $q_v \equiv 1 \text{ mod }p$ and $\overline{\rho}(\Frob_v)$ has order $p$. 

In \S \ref{sec_galois_theory}, we study the deformation theory of Galois representations, and in particular show that the Lie type elements of the dual Selmer group can be killed using a well-chosen local deformation problem. These results are then used in \S \ref{sec_r_equals_t} to prove an $R = \bbT$ theorem, which is very similar to the one proved in \cite[\S 6]{ThorneA}. In \S \ref{ordinary}, we prove an auxiliary automorphy liftings result for a unitary group in 4 variables, and apply this together with the techniques of \cite{Bar12} to construct ordinary automorphic liftings of automorphic exceptional residual representations. Finally, in \S \ref{sec_main_theorem}, we apply all of this to prove our main theorem (Theorem \ref{thm_main_theorem}).

\section{Notation and normalizations}\label{sec_normalizations}

\paragraph*{Galois groups.} A base number field $F$ having been fixed, we will also choose algebraic closures $\overline{F}$ of $F$ and $\overline{F}_v$ of $F_v$ for every finite place $v$ of $F$. The algebraic closure of $\bbR$ is $\bbC$. If $p$ is a prime, then we will also write $S_p$ for the set of places of $F$ above $p$, $\overline{\bbQ}_p$ for a fixed choice of algebraic closure of $\bbQ_p$, and $\val_p$ for the $p$-adic valuation on $\barqp$ normalized so that $\val_p(p) = 1$. These choices define the absolute Galois groups $G_F = \Gal(\overline{F}/F)$ and $G_{F_v} = \Gal(\overline{F}_v/F_v)$.  We write $I_{F_v} \subset G_{F_v}$ for the inertia subgroup. We also fix embeddings $\overline{F} \hookrightarrow \overline{F}_v$, extending the canonical embeddings $F \hookrightarrow F_v$. This determines for each place $v$ of $F$ an embedding $G_{F_v} \to G_{F}$. We write $\bbA_F$ for the adele ring of $F$, and $\bbA_F^\infty = \prod'_{v \nmid \infty} F_v$ for its finite part. If $v$ is a finite place of $F$, then we write $k(v)$ for the residue field at $v$, $\kappa(v)$ for the residue field of $\overline{F}_v$, and $q_v = \# k(v)$. If we need to fix a choice of uniformizer of $\cO_{F_v}$, then we will denote it $\varpi_v$.

If $S$ is a finite set of finite places of $F$, then we write $F_S$ for the maximal subfield of $\overline{F}$ unramified outside $S$, and $G_{F, S} = \Gal(F_S/F)$; this group is naturally a quotient of $G_F$. If $v \not\in S$ is a finite place of $F$, then the map $G_{F_v} \to G_{F, S}$ factors through the unramified quotient of $G_{F_v}$, and we write $\Frob_v \in G_{F, S}$ for the image of a \emph{geometric} Frobenius element. We write $\epsilon : G_F \to \bbZ_p^\times$ for the $p$-adic cyclotomic character; if $v$ is a finite place of $F$, not dividing $p$, then $\epsilon(\Frob_v) = q_v^{-1}$. If $\rho : G_F \to \GL_n(\barqp)$ is a continuous representation, we say that $\rho$ is de Rham if for each place $v | p$ of $F$, $\rho|_{G_{F_v}}$ is de Rham. In this case, we can associate to each embedding $\tau : F \hookrightarrow \barqp$ a multiset $\mathrm{HT}_\tau(\rho)$ of Hodge--Tate weights, which depends only on $\rho|_{G_{F_v}}$, where $v$ is the place of $F$ induced by $\tau$. This multiset has $n$ elements, counted with multiplicity. There are two natural normalizations for $\mathrm{HT}_\tau(\rho)$ which differ by a sign, and we choose the one with $\mathrm{HT}_\tau(\epsilon) = \{ -1 \}$ for every choice of $\tau$.

\paragraph*{Artin and Langlands reciprocity.} We use geometric conventions for the Galois representations associated to automorphic forms, which we now describe. First, we use the normalizations of the local and global Artin maps $\Art_{F_v} : F_v^\times \to W_{F_v}^\text{ab}$ and $\Art_F : \bbA_F^\times \to G_F^\text{ab}$ which send uniformizers to geometric Frobenius elements. If $v$ is a finite place of $F$, then we write $\rec_{F_v}$ for the local Langlands correspondence for $\GL_2(F_v)$, normalized as in Henniart \cite{Hen93} and Harris--Taylor \cite{Tay01} by a certain equality of $\epsilon$- and L-factors. We recall that $\rec_{F_v}$ is a bijection between the set of isomorphism classes of irreducible admissible representations $\pi$ of $\GL_2(F_v)$ over $\bbC$ and set of isomorphism classes of 2-dimensional Frobenius--semi-simple Weil--Deligne representations $(r, N)$ over $\bbC$. We define $\rec_{F_v}^T(\pi) = \rec_{F_v}(\pi \otimes | \cdot |^{-1/2})$. Then $\rec_{F_v}^T$ commutes with automorphisms of $\bbC$, and so makes sense over any field $\Omega$ which is abstractly isomorphic to $\bbC$ (e.g.\ $\barqp$).

If $v$ is a finite place of $F$ and $\chi : W_{F_v} \to \Omega^\times$ is a character with open kernel, then we write $\St_2(\chi \circ \Art_{F_v})$ for the inverse image under $\rec_{F_v}^T$ of the Weil--Deligne representation 
\begin{equation}\label{eqn_definition_of_steinberg_representation}
\left( \chi  \oplus  \chi | \cdot |^{-1} , \left(\begin{array}{cc} 0 & 1 \\ 0 & 0 \end{array}\right) \right).
\end{equation}
If $\Omega = \bbC$, then we call $\St_2 = \St_2(1)$ the Steinberg representation; it is the unique generic subquotient of the normalized induction $i_B^{\GL_2} |\cdot|^{1/2} \otimes |\cdot|^{-1/2}$. If $(r, N)$ is any Weil--Deligne representation, we write $(r, N)^\text{F-ss}$ for its Frobenius--semi-simplification. If $v$ is a finite place of $F$ and $\rho : G_{F_v} \to \GL_n(\barqp)$ is a continuous representation, which is de Rham if $v | p$, then we write $\mathrm{WD}(\rho)$ for the associated Weil--Deligne representation, which is uniquely determined, up to isomorphism. (If $v \nmid p$, then the representation $\mathrm{WD}(\rho)$ is defined in \cite[\S 4.2]{Tat79}. If $v | p$, then the definition of $\mathrm{WD}(\rho)$ is due to Fontaine, and is defined in e.g. \cite[\S 2.2]{Bre02}.)

\paragraph*{Automorphic representations.} In this paper, the only automorphic representations we consider are cuspidal automorphic representations $\pi = \otimes'_v \pi_v$ of $\GL_2(\bbA_F)$ such that for each $v | \infty$, $\pi_v$ is the lowest discrete series representation of $\GL_2(\bbR)$ of trivial central character.  (The only exception is in \S \ref{ordinary}, where it is necessary to use automorphic representations of $\GL_4$ in order to construct ordinary lifts of fixed 2-dimensional residual representations. Since the arguments in \S \ref{ordinary} are self-contained, we avoid introducing the relevant notation here.) In particular, any such $\pi$ is unitary. We will say that $\pi$ is a cuspidal automorphic representation of $\GL_2(\bbA_F)$ of weight 2. If $\pi$ is such a representation, then for every isomorphism $\iota : \barqp \to \bbC$, there is a continuous representation $r_\iota(\pi) : G_F \to \GL_2(\barqp)$ satisfying the following conditions:
\begin{enumerate}
\item The representation $r_\iota(\pi)$ is de Rham and for each embedding $\tau : F \hookrightarrow \barqp$, $\mathrm{HT}_\tau(\rho) = \{ 0, 1 \}$.
\item Let $v$ be a finite place of $F$. Then $\mathrm{WD}(r_\iota(\pi)|_{G_{F_v}})^\text{F-ss} \cong \rec^T_{F_v}(\iota^{-1} \pi_v)$.
\item Let $\omega_\pi$ denote the central character of $\pi$; it is a character $\omega_\pi : F^\times \backslash \bbA_F^\times \to \bbC^\times$ of finite order. Then $\det r_\iota(\pi) = \epsilon^{-1} \iota^{-1} (\omega_\pi \circ \Art_F^{-1})$.
\end{enumerate}
For concreteness, we spell out this local-global compatibility at the unramified places. Let $v \nmid p$ be a prime such that $\pi_v \cong i_B^{\GL_2} \chi_1 \otimes \chi_2$, where $\chi_1, \chi_2 : F_v^\times \to \bbC^\times$ are unramified characters and $i^{\GL_2}_B$ again denotes normalized induction from the upper-triangular Borel subgroup $B \subset \GL_2$. Then the representation $r_\iota(\pi)$ is unramified at $v$, and the characteristic polynomial of $\iota r_\iota(\pi)(\Frob_v)$ is $(X - q^{1/2} \chi_1(\varpi_v))(X - q^{1/2} \chi_2(\varpi_v))$. If $T_v$, $S_v$ are the usual unramified Hecke operators, and we write $t_v,$ $s_v$ for their respective eigenvalues on $\pi_v^{\GL_2(\cO_{F_v})}$, then we have
\begin{equation}\label{eqn_local_global_compatibility_at_unramified_places}
(X - q^{1/2} \chi_1(\varpi_v))(X - q^{1/2} \chi_2(\varpi_v)) = X^2 - t_v X + q_v s_v.
\end{equation}
With the above notations, the pair $(\pi, \omega_\pi)$ is a RAESDC automorphic representation in the sense of \cite{BLGGT}, and our representation $r_\iota(\pi)$ coincides with the one defined there. On the other hand, if $\sigma(\pi) : G_F \to \GL_2(\barqp)$ denotes the representation associated to $\pi$ by Carayol \cite{Car86}, then there is an isomorphism $\sigma(\pi) \cong r_\iota(\pi) \otimes (\iota^{-1} \omega_\pi \circ \Art_F^{-1})^{-1}.$

\paragraph*{Ordinary Galois and admissible representations.} Let $p$ be a prime, and let $K$ be a finite extension of $\bbQ_p$. If $\rho : G_K \to \GL_2(\overline{\bbQ}_p)$ is a de Rham representation with $\mathrm{HT}_\tau(\rho) = \{ 0, 1\}$ for all embeddings $\tau : K \hookrightarrow \overline{\bbQ}_p$, we say that $\rho$ is ordinary if it has the form
\[ \rho \sim \left( \begin{array}{cc} \psi_1 & \ast \\ 0 & \psi_2 \epsilon^{-1} \end{array} \right), \]
where $\psi_1, \psi_2 : G_K \to \overline{\bbQ}_p^\times$ are continuous characters which become unramified after a finite extension of $K$. (This is the same definition that appears in \cite[\S 5.5]{ThorneA}.) 

If $\pi$ is a cuspidal automorphic representation of $\GL_2(\bbA_F)$ of weight 2, $\iota : \overline{\bbQ}_p \cong \bbC$ is an isomorphism, and $v$ is a $p$-adic place of $F$, we say that $\pi_v$ is $\iota$-ordinary if it satisfies the condition in \cite[\S 4.1]{ThorneA}. We say that $\pi$ itself is $\iota$-ordinary if $\pi_v$ is $\iota$-ordinary for each $p$-adic place $v$ of $F$. One knows (\cite[Lemma 5.24]{ThorneA}) that $\pi_v$ is $\iota$-ordinary if and only if the Galois representation $r_\iota(\pi)|_{G_{F_v}}$ is $\iota$-ordinary.

\paragraph*{Rings of coefficients.} We will call a finite extension $E/\bbQ_p$ inside $\barqp$ a coefficient field. A coefficient field $E$ having been fixed, we will write $\cO$ or $\cO_E$ for its ring of integers, $k$ or $k_E$ for its residue field, and $\lambda$ or $\lambda_E$ for its maximal ideal. If $A$ is a complete Noetherian local $\cO$-algebra with residue field $k$, then we write $\ffrm_A \subset A$ for its maximal ideal, and $\CNL_A$ for the category of complete Noetherian local $A$-algebras with residue field $k$. We endow each object $R \in \CNL_A$ with its profinite ($\ffrm_R$-adic) topology. 

If $\Gamma$ is a profinite group and $\rho : \Gamma \to \GL_n(\barqp)$ is a continuous representation, then we can assume (after a change of basis) that $\rho$ takes values in $\GL_n(\cO)$, for some choice of coefficient field $E$. The semi-simplification of the composite representation $\Gamma \to \GL_n(\cO) \to \GL_n(k)$ is independent of choices, up to isomorphism, and we will write $\overline{\rho} : \Gamma \to \GL_n(\barfp)$ for this semi-simplification.

\paragraph*{Galois cohomology.} If $E$ is a coefficient field and $\overline{\rho} : \Gamma \to \GL_2(k)$ is a continuous representation, then we write $\ad \overline{\rho}$ for $\End_k(\overline{\rho})$, endowed with its structure of $k[\Gamma]$-module. We write $\ad^0 \overline{\rho} \subset \ad \overline{\rho}$ for the submodule of trace 0 endomorphisms, and (if $\Gamma = G_F$) $\ad^0 \overline{\rho}(1)$ for its twist by the cyclotomic character. If $M$ is a discrete $\bbZ[G_F]$-module (resp. $\bbZ[G_{F, S}]$-module), then we write $H^1(F, M)$ (resp. $H^1(F_S/F, M)$) for the continuous Galois cohomology group with coefficients in $M$. Similarly, if $M$ is a discrete $\bbZ[G_{F_v}]$-module, then we write $H^1(F_v, M)$ for the continuous Galois cohomology group with coefficients in $M$. If $M$ is a discrete $k[G_F]$-module (resp. $k[G_{F, S}$]-module, resp. $k[G_{F_v}]$-module), then $H^1(F, M)$ (resp. $H^1(F_S/F, M)$, resp. $H^1(F_v, M)$) is a $k$-vector space, and we write $h^1(F, M)$ (resp. $h^1(F_S/F, M)$, resp. $h^1(F_v, M)$) for the dimension of this $k$-vector space, provided that it is finite.

\section{Exceptional representations}\label{examples}

Let $F$ be a totally real number field. We say that a representation $\rhobar:G_F \ra \GL_2(\overline{\bbF}_5)$ is exceptional if its projective image $\ad^0 \overline{\rho}(G_F)$ is isomorphic to $\PGL_2(\F_5)$ and the fixed field of $\ker \ad^0 \overline{\rho}$ contains $F(\zeta_5)$. This forces $\overline{\rho}$ to be totally odd (i.e. $\det \overline{\rho}(c) = -1$ for every complex conjugation $c \in G_F$.) We recall that the classification of finite subgroups of $\PGL_2(\overline{\bbF}_5)$ shows that $\PGL_2(\F_5) \cong S_5$, and that any subgroup $H \subset \PGL_2(\overline{\bbF}_5)$ such that $H \cong S_5$ is a conjugate of $\PGL_2(\F_5)$.

The main point of this paper is to prove automorphy lifting theorems for geometric Galois representations $\rho : G_F \to \GL_2(\overline{\bbQ}_5)$ with exceptional residual representation. Before doing this, we make some general remarks.
\begin{lemma}\label{bc}
 Let  $\rhobar:G_F \rightarrow \GL_2(\overline \F_5)$ be any continuous representation. Then:
\begin{enumerate} \item  Let $F'/F$ be a totally real number field. Then $\rhobar|_{F'}$ is  exceptional if and only if $\rhobar$ is exceptional. 
\item There is no conjugate  of $\rhobar$   valued in $\GL_2(\F_5)$.
\end{enumerate}
\end{lemma}
\begin{proof}
\begin{enumerate} \item We note that the quadratic subfield of the extension of $F$ cut out by the projective representation arising from an exceptional  $\rhobar$ is not totally real. Then the proof follows easily on using that a normal subgroup  $H$ of $S_5$ with non-trivial  image in $\Z/2\Z=S_5/A_5$ is  all of $S_5$.
\item For any representation $\tau:G_F \ra \GL_2(\F_5)$, the image of any complex conjugation in $G_F$  maps to $\PSL_2(\F_5)$.
\end{enumerate}
\end{proof}
We would like a construction  of exceptional $\rhobar$, equivalently $S_5$ extensions of $F$, with the corresponding quadratic subfield  fixed by $A_5$ equal to $F(\zeta_5)$. Serre in a message to the first-named author told us about a construction of Mestre, which for any number field $F$ and a quadratic extension $F'/F$,  produces an $S_n$ extension $L/F$ containing $F'$.  Mestre has made his construction explicit for $n=5$,  and the quadratic extension $F(\zeta_5)/F$. We now describe this construction.

Consider  $P =x(x^4+B)\in F[x]$, with non-zero discriminant $\Delta(P) =2^8 B^5$,  and  the one-parameter family  $ E(T) =P-TQ  $ where
$Q =  20^2 B^2 x^4 + 12^2 B^3$. The discriminant of is $\Delta(E(T))=\Delta(P)S(T)^2$, where $S(T) \in F(T)$. The Galois group of the splitting field of $E(T)$ over $F(T)$ is either $A_5$ or $S_5$ depending on whether $\Delta(P) \in F^*$ is a square or not (in the latter case it is not a regular extension and contains the extension $F(\sqrt{\Delta(P)})$). The reader can consult \S 9.3 of \cite{Ser92} for a related construction.

We now fix $B$ so that $F(\zeta_5)=F(\sqrt{B})$ and choose a generic $T \in F$ using the Hilbert irreducibility theorem. This gives examples of exceptional $\rhobar$. As an explicit example, one can take (over any totally real number field $F$ with $\sqrt{5} \in F$) the splitting field over $F$ of the polynomial
\[ x^5 + (1000\sqrt{5} + 3000)x^4 -(\sqrt{5} + 5)x/2 - (1440 \sqrt{5} + 3600). \]

\section{Some group theory}\label{group_theory}

\subsection{A matrix lemma}

We need the following elementary lemma about almost diagonalizable matrices. Let $E$ be a coefficient field.
\begin{lemma}\label{eigenvalues}
Let $A \in \End_\cO(M)$, where $M = (\cO/\lambda^n)^2$, and let $f_A(X)$ denote the characteristic polynomial of $A$. Suppose that there exist $\alpha, \beta \in \cO/\lambda^n$ such that $f_A(X) = (X - \alpha)(X - \beta)$ and $\alpha-\beta$ is divisible exactly by $\lambda^m$ for some integer $0 \leq m \leq n-1$. Then there exist $e_1,e_2 \in M$ with $\lambda^{n-m-1}e_1 \neq 0$ and $\lambda^{n-m-1} e_2 \neq 0$ such that $Ae_1=\alpha e_1,Ae_2=\beta e_2$.
\end{lemma}
\begin{proof}
It suffices to show that $\lambda^{n-m-1}(A - \beta)M \neq 0$. If $\lambda^{n-m -1}(A - \beta)M = 0$, then $(A - \beta)$ acts as 0 on $\lambda^{n-m-1} M \cong (\cO / \lambda^{m+1})^2$, implying that $f_A(X) \equiv (X - \beta)^2 \text{ mod }\lambda^{m+1}$, hence $\alpha \equiv \beta \text{ mod }\lambda^{m+1}$. This contradicts our hypothesis.
\end{proof}

\subsection{Closed subgroups of $\PGL_2(\cO)$}

Now let $p = 5$, and let $E$ be a coefficient field. The following lemma plays a key role in this paper.
\begin{lemma}\label{Z5}
Let $H$ be closed subgroup of $\PGL_2(\Oc)$ with residual image a conjugate $\PGL_2(\F_5) \subset \PGL_2(k)$. Then $H$ is  conjugate to $\PGL_2(A)$ where  $A$ is a closed $\Z_5$-subalgebra  of $\Oc$. In particular  $H$ contains a conjugate of $\PGL_2(\Z_5)$, and the maximal abelian quotient of $H$ is finite, and in fact isomorphic to $\Z/2\Z\simeq \PGL_2(\F_5)/\PSL_2(\F_5)$.
\end{lemma}
\begin{proof}
The proof follows from the  arguments in \cite{Boe} as explained to us by  G. B\"ockle. After conjugating $H$, we can assume that the residual image of $H$ is equal to $\PGL_2(\bbF_5)$. We denote by $\Ad$ the adjoint representation of $\PGL_2(\F_5)$ on $M_2(\F_5)$ and by $Z$ the subspace of scalar matrices. We note the following key facts:
\begin{enumerate} \item $H^1(\PGL_2(\F_5),\Ad/Z)=0$.
\item The map $\PGL_2(\Z/5^2\Z) \ra \PGL_2(\Z/5\Z)$ does not split.
\item The representation $\Ad/Z \cong \Ad^0$ is absolutely irreducible as $\bbF_5[\PGL_2(\F_5)]$-module.
\end{enumerate}
Consider the ring $R = \Z_5+\lambda \Oc$. Then $R \in \CNL_{\bbZ_5}$, and $H \subset \PGL_2(R)$. Let $\rho : H \to \PGL_2(R)$ denote the tautological representation, $\overline{\rho} : H \to \PGL_2(\bbF_5)$ its residual representation. We can define a functor $\mathrm{Def}_{\overline{\rho}} : \CNL_{\bbZ_5} \to \mathrm{\Sets}$ by associating to each $A \in \CNL_{\bbZ_5}$ the set of strict equivalence classes of liftings $r : H \to \PGL_2(A)$ of $\overline{\rho}$. The group $H$ satisfies Mazur's condition $\Phi_p$ \cite{Maz}, so the functor $\mathrm{Def}_{\overline{\rho}}$ is represented by an object $\cR \in \CNL_{\bbZ_5}$. Let $\rho^u : H \to \PGL_2(\cR)$ denote a representative of the universal deformation.

We claim that $\rho^u$ is surjective. This claim implies the lemma: by universality, there is a canonical classifying homomorphism $\cR \to R$. Let $A \subset R$ denote the image. Then $\rho$ is conjugate inside $\PGL_2(\cO)$ to the group $\PGL_2(A)$. The natural map $\bbZ_5 \to A$ is injective, so we find that $H$ contains a conjugate of $\PGL_2(\bbZ_5)$. To calculate the maximal abelian quotient, we first note that the commutator subgroup of $\SL_2(A)$ is $\SL_2(A)$.
Therefore the commutator subgroup of $\PGL_2(A)$ contains the image of $\SL_2(A)$.
This shows that the abelianization of $\PGL_2(A)$ is the quotient of $\GL_2(A)$ modulo $A^*\SL_2(A)$.
Via the determinant map whose kernel is $\SL_2(A)$, this quotient is identified with $A^*/A^{*2}$. 
Because $A$ is local with residue characteristic different from 2, $A^{*2}$ contains $(1+{\mathfrak m}_A)^2=1+{\mathfrak m}_A$. Thus  the commutator subgroup of $\PGL_2(A)$ is $\PSL_2(A)$, and $\PGL_2(A)/\PSL_2(A)\simeq \PGL_2(\F_5)/\PSL_2(\F_5)$,  proving the last sentence of the lemma.

To prove the claim, it is enough (cf. \cite[Theorem 1.6]{Boe}, \cite[Proposition 2]{Maz86}) to show that the induced map
\[ \rho^u \text{ mod } \ffrm_{\cR}^2 : H \to \PGL_2(\cR / \ffrm_{\cR}^2) \]
is surjective. We will prove this using the facts 1, 2, 3 above. The target of this map sits in a short exact sequence
\[ \xymatrix@1{ 1 \ar[r] & \ffrm_{\cR} / \ffrm^2_{\cR} \otimes_{\bbF_5} \Ad^0 \ar[r] &  \PGL_2(\cR / \ffrm_{\cR}^2) \ar[r] & \PGL_2(\bbF_5) \ar[r] & 1.} \]
We know that the image surjects to $\PGL_2(\bbF_5)$, so it is enough to show that the image contains $\ffrm_{\cR} / \ffrm_{\cR}^2 \otimes_{\bbF_5} \Ad^0$. Because of fact 3, the image has the form $V \otimes_{\bbF_p} \Ad^0$ for some subgroup $V \subset \ffrm_{\cR} / \ffrm_{\cR}^2$. If $V \neq \ffrm_{\cR} / \ffrm_{\cR}^2$, then we can choose a codimension 1 subspace $K \subset \ffrm_{\cR} / \ffrm_{\cR}^2$ containing $V$. The subspace $K$ is then automatically an ideal. Let $A = \cR / (\ffrm_{\cR}^2, K)$. We can then pushout by the surjection $\cR \to A$ to obtain a short exact sequence
\[ \xymatrix@1{ 1 \ar[r] & \Ad^0 \ar[r] &  \PGL_2(A) \ar[r] & \PGL_2(\bbF_5) \ar[r] & 1.} \]
The ring $A$ is isomorphic either to $\bbZ_5 / 25 \bbZ_5$ or $\bbF_5[\epsilon]$. By construction, the image $H_A$ of $H$ in the group $\PGL_2(A)$ has trivial intersection with $\Ad^0$, and therefore describes a section of the map $\PGL_2(A) \to \PGL_2(\bbF_5)$.

If $A = \bbZ_5 / 25 \bbZ_5$, then this contradicts fact 2, i.e.\ that the reduction map does not split. If $A = \bbF_5[\epsilon]$, then fact 1 implies that $A = M^{-1} \PGL_2(\bbF_5) M$ for some $M \in 1 + \epsilon \Ad^0$. By universality, the map $\cR \to A$ then factors as $\cR \to \bbF_5 \to A$, which contradicts that $\cR \to A$ is surjective, by construction. In either case, then, we obtain a contradiction. This establishes the claim, and therefore the lemma.
\end{proof}
Note that the lemma above is false if the residual image is $\PSL_2(\F_5)$, as $H$ could be isomorphic to $A_5$.

\section{Shimura curves and Hida varieties}\label{sec_shimura_curves_and_hida_varieties}

In the proof of our main theorem, we need to make use of level raising and lowering results in a similar way to \cite{ThorneA}. In this section, we recall some of the key results of \cite[\S 4]{ThorneA} in this direction, and indicate any needed extensions. 

Let $F$ be a totally real number field of even degree over $\bbQ$, and let $Q$ be a finite set of finite places of $F$. We fix for each such $Q$  the quaternion algebra $B_Q$ as in \cite[\S 4]{ThorneA}, ramified at $Q$ and some set of infinite places, a maximal order $\Oc_Q$ of $B_Q$, and write $G_Q$ for the associated reductive group over $\Oc_F$. We fix a choice of an auxiliary place $a$ with $q_a>4^d$ and define the class $\cJ_Q$ of good open compact subgroups $U=\prod_v U_v$  of $G_Q(\A_F^\infty)$ (which depends on $a$).

Depending on whether $Q$ is even or odd, we have for each $U \in \cJ_Q$ the Hida variety $X_Q(U)$ or the Shimura curve $M_Q(U)$ defined over $F$  as in \cite[\S \S 4.2, 4.3]{ThorneA}. We have a theory of integral models of the $M_Q(U)$ for good $U$ as summarized in  \cite[\S \S 4.4, 4.5]{ThorneA}.

Fix a prime $p$ and a coefficient field $E$. We now skip to \cite[\S 4.6]{ThorneA} and for a set of places $Q$ of $F$ not containing $a$, define a collection of Hecke algebras and modules for these algebras with $\Oc$-coefficients. Let $S$ be a  finite set of finite places of  $F$ containing $Q$. Write $\T^{S,\text{univ}}$ and $\T_Q^{S,\text{univ}}$ for the universal Hecke algebras as in \cite{ThorneA}. The superscript $S$ denotes that we are including the unramified Hecke operators $T_v$ only for primes $v \not\in S$, and the subscript $Q$ indicates that we are including the $U_v$ operators for $v \in Q$.

Let $U \in {\mathcal J}_Q$.  If $\# Q$ is odd we define $H_Q(U)=H^1(M_Q(U)_{\overline F},\Oc)$. If $\# Q$ is even,  then we define $H_Q(U)=H^0(X_Q(U),\Oc)$. In either case, $H_Q(U)$ is a finite free $\Oc$-module. If $S$ contains the places of $F$ such that $U_v \neq \GL_2(\Oc_v)$, then the Hecke algebra $\T_Q^{S,\text{univ}}$ acts on $H_Q(U)$.

If $\# Q$ is odd, then the Galois group $G_{F}$ acts on $H_Q(U)$, and this action commutes with the action of $\T_Q^{S,\text{univ}}$. The Eichler--Shimura relation holds: for all finite places $v \notin S \cup S_p$, the action of $G_{F_v}$ is unramified and we have the relation $$\Frob_v^2-S_v^{-1}T_v\Frob_v+qS_v^{-1}=0$$ in ${\rm End}_\Oc(H_Q(U))$. 

\begin{theorem}\label{ES}
Let $\# Q$ be odd and  let ${\mathfrak m}\subset \T^S(H_Q(U))$ be a maximal ideal.
\begin{enumerate} \item There exists a continuous representation $\rhobar_{\mathfrak m} : G_F \to \GL_2(\T^S(H_Q(U))/{\mathfrak m})$ satisfying the following condition: for all finite places $v \not\in S \cup S_p$ of $F$, $\rhobar_{\mathfrak m}$ is unramified at $v$, and the characteristic polynomial of $\rhobar_{\mathfrak m}(\Frob_v)$ is given by $X^2 - T_v X + q_v S_v \text{ mod }{\mathfrak m}$.
\item Suppose that the representation $\rhobar_{\mathfrak m}$ is absolutely irreducible (in other words, the ideal ${\mathfrak m}$ is \emph{non-Eisenstein}). Then there exists a continuous representation $\rho_{\mathfrak m}:G_F \ra \GL_2(\T^S(H_Q(U))_{\mathfrak m})$ lifting $\rhobar_{\mathfrak m}$, and satisfying the following condition: for all finite places $v\not\in S \cup S_p$ of $F$, $\rho_{\mathfrak m}$ is unramified at $v$, and the characteristic polynomial of $\rhobar_{\mathfrak m}(\Frob_v)$ is given by $X^2 - T_v X + q_v S_v$. 
\item Suppose that ${\mathfrak m}$ is non-Eisenstein. There is a finite $\T^S(H_Q(U))_{\mathfrak m}$-module $M$, together with an isomorphism of $\T^S(H_Q(U))[G_F]$-modules
\[H_Q(U) \cong \rho_{\mathfrak m} \otimes_{\T^S(H_Q(U))_{\mathfrak m}} (\epsilon\det \rho_{\mathfrak m}) \otimes_{\T^S(H_Q(U))_{\mathfrak m}} M.\]
\end{enumerate}
\end{theorem}
\begin{proof}
This is \cite[Proposition 4.7]{ThorneA}.
\end{proof}

\subsection{Level-raising}\label{raising}

We now discuss level-raising, as in \cite[\S 4.8]{ThorneA}. We have to make slight adjustments as we want to allow the set $Q$ considered there to contain places $w$ such that $q_w \equiv 1 \text{ mod }p$. 
We suppose given the following data:
\begin{itemize}
\item  A finite set $R$ of finite places of $F$ of even cardinality, disjoint from $S_p \cup \{a\}$.
\item  A good subgroup $U=\prod_w U_w \in \mathcal J_R$. 
\item A finite set $Q$ of finite places of $F$, of even cardinality, and satisfying the following conditions:
\begin{itemize}
\item $U_w=\GL_2(\Oc_w)$ for $w \in Q$.
\item $Q \cap (S_p \cup \{a\} \cup R)=\emptyset$.
\end{itemize}
\end{itemize}
If $J \subset Q$, then we define the subgroup $U_J$ of $G_{R\cup J}(\A_F^\infty)$ as in \cite[\S 4.7]{ThorneA}:
\begin{itemize}
\item If $w \not\in J$, then $U_{J,w}=U_w$.
\item If $w \in J$, then $U_{J,w}$ is the unique maximal compact subgroup of $G_{R \cup J}(F_w)$.
\end{itemize}
Let $S$ be a finite set of finite places of $F$ containing $S_p\cup  R \cup Q$ and the  places $w$ such that $U_w \neq \GL_2(\Oc_{F_w})$. Let $\mathfrak m ={\mathfrak m}_\emptyset$ be a non-Eisenstein maximal ideal of $\T^{S,\univ} = \T^{S,\univ}_\emptyset$ in the support of $H_R(U)$. Thus $\rhobar_{\mathfrak m}$ is absolutely irreducible  and for each $v \in Q$, $\rhobar_{\mathfrak m}|_{G_{F_v}}$ is unramified.  After enlarging the coefficient field $E$,  we can assume that for all $v \in Q$, the eigenvalues $\alpha_v,\beta_v$ of $\rhobar_{\mathfrak m}(\Frob_v)$ lie in $k$. We demand that:
\begin{itemize}
\item   For $v \in Q$, $\alpha_v / \beta_v = q_v$.
\item   If $v \in Q$ and $q_v \equiv 1 \text{ mod }p$, then $\rhobar(\Frob_v)$ is unipotent of order $p$. 
\end{itemize}
 For each $J \subset Q$,  let ${\mathfrak m}_J \subset \T_J^{S,\univ}$ be the maximal ideal generated by $\mathfrak m_\emptyset$ and $U_v-\alpha_v$, $v \in J$. Note that for $v \in J$ with $q_v \equiv -1 \text{ mod }p$, this represents a choice of one of the eigenvalues of $\overline{\rho}(\Frob_v)$ (since $\beta_v$ would also be allowable).
\begin{lemma}\label{raising1}
 The ideal ${\mathfrak m}_Q$ is in the support of $H_{R \cup Q}(U_Q)$.
\end{lemma}
\begin{proof}
The statement is the same as \cite[Lemma 4.11]{ThorneA}, except that we now allow also places $v \in Q$ such that $q_v \equiv 1 \text{ mod }p$. The proof goes through unchanged in this case (and is slightly easier than the $q_v \equiv -1 \text{ mod }p$ case, since there is only one possibility for the choice of $\alpha_v$).
\end{proof}
We continue with the same notation and assumptions, and prove:
\begin{prop}\label{bound}
We have 
\[ 1\leq {\rm dim}_k(H_{R \cup Q}(U_Q) \otimes_\Oc k)[{\mathfrak m}_Q] \leq 4^{\#Q}
{\rm dim}_k(H_{R}(U_\emptyset) \otimes_\Oc k )[{\mathfrak m}_\emptyset]. \]
\end{prop}
\begin{proof}
The statement is the same as \cite[Proposition 4.12]{ThorneA}, except that we again allow here places $v \in Q$ with $q_v \equiv 1 \text{ mod }p$. The lower bound is the content of Lemma \ref{raising1}. For the upper bound, it suffices by induction to establish the statement
\[ \dim_k (H_{R \cup J}(U_J) \otimes_\Oc k)[{\mathfrak m}_J] \leq 4
\dim_k(H_{R \cup \overline{J}}(U_{\overline{J}}) \otimes_\Oc k )[{\mathfrak m}_{\overline{J}}], \]
where $J \subset Q$ is a non-empty subset and $\overline{J} = J - \{ v \}$ for some $v \in J$. If $q_v \not\equiv 1 \text{ mod }p$, then this is done in the proof of \cite[Proposition 4.12]{ThorneA}. It therefore remains to treat the case where $q_v \equiv 1 \text{ mod }p$.

Let us first treat the case where $\#\overline J$ is odd. Then, just as in \emph{loc. cit.}, we derive an exact sequence of $k[G_{F_v}]$-modules
\[ \xymatrix@1{ 0\ar[r] & (H_{R \cup J}(U_J) \otimes_\cO k)[\ffrm_J] \ar[r] & (H_{R \cup \overline{J}}(U'_{\overline{J}})_{\ffrm_J} \otimes_\cO k)[\ffrm_J] \ar[r] & (H_{R \cup \overline{J}}(U_{\overline{J}})_{\ffrm_{\overline{J}}}^2 \otimes_\cO k)[\ffrm_{\overline{J}}]}, \]
where $\Frob_v^{-1}$ acts as $U_v$ on the first term $(H_{R \cup J}(U_J) \otimes_\cO k)[\ffrm_J]$, hence as the identity. Moreover, $(H_{R \cup \overline{J}}(U'_{\overline{J}})_{\ffrm_J} \otimes_\cO k)[\ffrm_J]$ is isomorphic to a direct sum of copies of $\overline{\rho}_\ffrm$ as $k[G_F]$-module. Since $\overline{\rho}(\Frob_v)$ is supposed unipotent non-trivial, we can find a $k[G_{F_v}]$-submodule $A \subset (H_{R \cup \overline{J}}(U'_{\overline{J}})_{\ffrm_J} \otimes_\cO k)[\ffrm_J]$ such that
\[ (H_{R \cup \overline{J}}(U'_{\overline{J}})_{\ffrm_J} \otimes_\cO k)[\ffrm_J] = A \oplus (\Frob_v - 1)A, \]
and the kernel of $\Frob_v - 1$ on $A$ is trivial. Using the above exact sequence we thus get
\[ \dim_k (H_{R \cup J}(U_J) \otimes_\cO k)[\ffrm_J] \leq \dim_k (\Frob_v - 1)A = \dim_k A \leq \dim_k (H_{R \cup \overline{J}}(U_{\overline{J}})_{\ffrm_{\overline{J}}}^2 \otimes_\cO k)[\ffrm_{\overline{J}}]. \]
This completes the proof in the case that $\#\overline J$ is odd. Now suppose that $\# \overline{J}$ is even. In this case, we can construct as in the proof of  \cite[Proposition 4.12]{ThorneA} a map
\[ (H_{R \cup J}(U_J)_{\ffrm_J} \otimes_\cO k)[\ffrm_J] \to (H_{R \cup \overline{J}}(U_{\overline{J}})_{\ffrm_{\overline{J}}} \otimes_\cO k)^2[\ffrm_J], \]
with kernel contained in the submodule
\[ (H_{R \cup J}(U_J)^{I_{F_v}}_{\ffrm_J} \otimes_\cO k) [\ffrm_J] \subset (H_{R \cup J}(U_J)_{\ffrm_J} \otimes_\cO k)[\ffrm_J], \]
on which $\Frob_v^{-1}$ acts trivially. Moreover, $(H_{R \cup J}(U_J)_{\ffrm_J} \otimes_\cO k)[\ffrm_J]$ is isomorphic as $k[G_{F}]$-module to a direct sum of copies of $\overline{\rho}_\ffrm$. Again using the fact that $\overline{\rho}_\ffrm(\Frob_v)$ is unipotent, we can find a $k[G_{F_v}]$-submodule $A \subset (H_{R \cup J}(U_J)_{\ffrm_J} \otimes_\cO k)[\ffrm_J]$ on which $\Frob_v - 1$ is injective and such that $(H_{R \cup J}(U_J)_{\ffrm_J} \otimes_\cO k)[\ffrm_J] = A \oplus (\Frob_v - 1)A$. Then $A$ has trivial intersection with the kernel of the above map, and we obtain
\[ \begin{split} \dim_k (H_{R \cup J}(U_J)_{\ffrm_J} \otimes_\cO k)[\ffrm_J] = 2 \dim_k A & \leq 2 \dim_k (H_{R \cup \overline{J}}(U_{\overline{J}})_{\ffrm_{\overline{J}}} \otimes_\cO k)^2[\ffrm_J]\\ & \leq 4 \dim_k (H_{R \cup \overline{J}}(U_{\overline{J}})_{\ffrm_{\overline{J}}} \otimes_\cO k)[\ffrm_{\overline{J}}],\end{split} \]
as desired.
\end{proof}

\subsection{Level lowering mod $p^N$}\label{lowering}
We now state an analogue of \cite[Theorem 4.14]{ThorneA} which covers the cases relevant for this paper.  We fix again a finite set $R$ of finite places of $F$ of even cardinality, disjoint from $S_p \cup \{ a \}$ and a good subgroup $U \in \cJ_R$. If $Q$ is any finite set of finite places of $F$, disjoint from $S_p \cup \{ a \} \cup R$ such that for each $v \in Q$, $U_v = \GL_2(\cO_{F_v})$, then we define $U_Q \in \cJ_{R \cup Q}$ by replacing $U_v$ for $v \in Q$ by the unique maximal compact subgroup of $G_{R \cup Q}(F_v)$ (cf. \cite[\S 4.9]{ThorneA}).

Let $S$ be a finite set of finite places of $F$, containing $S_p$, such that $U_v = \GL_2(\cO_{F_v})$ for all $v \not\in S$. Let $\ffrm \subset \bbT^{S, \text{univ}}$ be a non-Eisenstein maximal ideal which is in the support of $H_R(U)$.
\begin{theorem}\label{thm_level_lowering_mod_p_to_the_N}
Let $m, r \geq 0$ be  fixed integers and consider an integer $N \geq 2mr$, and a lifting of $\rhobar_{\mathfrak m}$ to a continuous representation $\rho:G_F \ra \GL_2(\Oc/\lambda^N)$. We assume that $\rho$ satisfies the following properties:
\begin{enumerate} \item $\rho$ is unramified outside $S$.
\item There exists a set $Q$ as above of even cardinality  and a homomorphism $f:\T^{S \cup Q}_Q(H_{R \cup Q}(U_Q)) \ra \Oc/\lambda^N$ satisfying:
\begin{enumerate} \item For each $v \in Q$, $q_v \not\equiv 1 \text{ mod }\lambda^{m+1}$. There are exactly $r$ places of $Q$  such that $q_v \equiv 1 \text{ mod }p$.
\item For every $v \notin S \cup Q$ we have $f(T_v)=\tr \rho(\Frob_v)$.
\item Let $I=\ker(f)$. Then $(H_{R \cup Q}(U_Q) \otimes_\Oc\Oc/\lambda^N)[I]$ contains a submodule isomorphic to $\Oc/\lambda^N$.
\end{enumerate}
\end{enumerate}
Then there exists a homomorphism $f':\T^{S \cup Q}(H_R(U)) \ra \Oc/\lambda^{N-2mr}$ such that for all $v \notin S \cup Q$, we have $f'(T_v)=\tr \rho(\Frob_v)$. 
\end{theorem}
(In the applications of this result in this paper, we will take $m = 1$.) Unlike in the previous section, we do not need the image of $\overline{\rho}_\ffrm(\Frob_v)$ to be  of order $p$  at those places $v \in Q$ such that $q_v \equiv 1 \text{ mod }p$. Just as in \cite[\S 4.9]{ThorneA}, this theorem has the following corollary.
\begin{cor}\label{approximation}
 Let $\rho:G_F \ra \GL_2(\Oc)$ be a continuous lifting of $\rhobar_{\mathfrak m}$ unramified outside $S$, and let $m, r \geq 0$ be integers. Suppose that for every $N \geq 2mr$ there exists a set $Q$ as above, of even cardinality, and a homomorphism $f:\T^{S \cup Q}_Q(H_{R \cup Q}(U_Q)) \ra \Oc/\lambda^N$ satisfying:
\begin{enumerate} \item For each $v \in Q$, $q_v \not\equiv 1 \text{ mod }\lambda^{m+1}$. There are exactly $r$ places of $Q$  such that $q_v \equiv 1 \text{ mod }p$.
\item For every finite place $v \notin S \cup Q$ of $F$, we have $f(T_v)=\tr \rho(\Frob_v)$.
\item Let $I=\ker(f)$. Then $(H_{R \cup Q}(U_Q)\otimes \Oc/\lambda^N))[I]$ contains a submodule isomorphic to $\Oc/\lambda^N$.
 \end{enumerate}
 Then $\rho$ is automorphic: there exists a cuspidal automorphic representation $\pi$ of $\GL_2(\A_F)$ of weight 2, and an isomorphism $\rho \otimes \aQp \cong r_\iota(\pi)$.
\end{cor}
Just as in  \cite[\S 4.9]{ThorneA}, Theorem \ref{thm_level_lowering_mod_p_to_the_N} follows by induction from the following result.
\begin{prop}\label{prop_inductive_step_in_mod_p_n_level_lowering}
Let $m \geq 0$ be an integer and consider an integer $N \geq 1$ and a lifting of $\rhobar_{\mathfrak m}$ to a continuous representation $\rho:G_F \ra \GL_2(\Oc/\lambda^N)$. We assume that $\rho$ satisfies the following properties:
\begin{enumerate} \item $\rho$ is unramified outside $S$.
\item There exists a set $Q$ as above of even cardinality  and a homomorphism $f:\T^{S \cup Q}_Q(U_Q) \ra \Oc/\lambda^N$ satisfying:
\begin{enumerate} \item For each $v \in Q$, $q_v \not\equiv 1 \text{ mod }\lambda^{m+1}$.
\item For every place $v \notin S \cup Q$ we have $f(T_v)=\tr \rho(\Frob_v)$.
\item Let $I=\ker(f)$. Then $(H_{R \cup Q}(U_Q) \otimes_\Oc\Oc/\lambda^N)[I]$ contains a submodule isomorphic to $\Oc/\lambda^N$.
\end{enumerate}
\end{enumerate}
Choose $v \in Q$, and let $\overline Q=Q - \{v\}$. Then there exists a homomorphism $f':\T^{S \cup Q}_{\overline Q}(H_{R\cup \overline Q}(U_{\overline Q})) \ra \Oc/\lambda^{N-2m}$ such that for all $v \notin S \cup Q$, we have $f'(T_v)=\tr \rho(\Frob_v)$. Moreover writing $I'=\ker f'$, $(H_{R \cup \overline Q}(U_{\overline Q}) \otimes_\Oc\Oc/\lambda^N)[I']$ contains a submodule isomorphic to $\Oc/\lambda^{N-2m}$.
\end{prop}
\begin{proof}
If $q_v \not\equiv 1 \text{ mod }p$, then this is \cite[Proposition 4.16]{ThorneA}. We treat the case where $q_v \equiv 1 \text{ mod }p$, splitting into cases according to the parity of $\# Q$. Let us first suppose that $\# Q$ is even. Let $\ffrm_Q \subset \bbT_Q^{S \cup Q, \text{univ}}$ denote the ideal generated by $I$ and $\lambda$. Just as in the proof of \cite[Proposition 4.16]{ThorneA}, we obtain a short exact sequence
\[ 0 \to (H_{R \cup Q}(U_Q)_{\ffrm_Q} \otimes_\cO \cO/\lambda^N)[I] \to (H_{R \cup \overline{Q}}(U_{\overline{Q}}')_{\ffrm_Q} \otimes_\cO \cO/\lambda^N)^{I_{F_{v}}}[I] \to (H_{R \cup \overline{Q}}(U_{\overline{Q}})^2_{\ffrm_Q} \otimes_\cO \cO/\lambda^N)[I], \]
an $\cO/\lambda^N$-module $U_0$, and an isomorphism $(H_{R \cup \overline{Q}}(U_{\overline{Q}}')_{\ffrm_Q} \otimes_\cO \cO/\lambda^N)[I] \cong U_0 \otimes_\cO \rho$. Moreover, $\Frob_v^{-1}$ acts as the scalar $f(U_v)$ on the first term $(H_{R \cup Q}(U_Q)_{\ffrm_Q} \otimes_\cO \cO/\lambda^N)[I]$ in the above exact sequence. By Lemma \ref{eigenvalues}, there exists an element $e \in (H_{R \cup \overline{Q}}(U_{\overline{Q}}')_{\ffrm_Q} \otimes_\cO \cO/\lambda^N)[I]$ such that $\lambda^{N-m-1} e \neq 0$ and $\Frob_v^{-1} e = q_v^{-1} f(U_v) e$.

The intersection of $\cO \cdot e$ with $(H_{R \cup Q}(U_Q)_{\ffrm_Q} \otimes_\cO \cO/\lambda^N)[I]$ is killed by $\lambda^m$, so the image of $e$ in $(H_{R \cup \overline{Q}}(U_{\overline{Q}})^2_{\ffrm_Q} \otimes_\cO \cO/\lambda^N)[I]$ generates an $\cO$-submodule $A$ such that $\lambda^{N - 2m - 1} A \neq 0$. The proof is completed in this case by taking $f'$ to be the homomorphism $\bbT^{S \cup Q}_{\overline{Q}} \to \cO / \lambda^{N - 2m}$ which is associated to $A$.

We now treat the case where $\# Q$ is odd. Let $\ffrm_Q$ be as before. Just as in the proof of \cite[Proposition 4.16]{ThorneA}, we obtain a morphism of $\bbT_Q^{S \cup Q, \text{univ}}[G_{F_v}]$-modules
\[ (H_{R \cup Q}(U_Q) \otimes_\cO \cO/\lambda^N)_{\ffrm_Q}[I] \to (H_{R \cup \overline{Q}}(U_{\overline{Q}}) \otimes_\cO \cO/\lambda^N)_{\ffrm_Q}^2[I], \]
with kernel contained in the submodule
\[ (H_{R \cup Q}(U_Q)^{I_{F_{v}}} \otimes_\cO \cO/\lambda^N)_{\ffrm_Q}[I] \subset (H_{R \cup Q}(U_Q) \otimes_\cO \cO/\lambda^N)_{\ffrm_Q}[I],  \]
on which $\Frob_v^{-1}$ acts by $f(U_v) \in \cO/\lambda^N$. There is also an $\cO/\lambda^N$-module $U_0$ and an isomorphism
\[(H_{R \cup Q}(U_Q) \otimes_\cO \cO/\lambda^N)[I] \cong U_0 \otimes_\cO \rho. \]
By Lemma \ref{eigenvalues}, we can find an element $e \in (H_{R \cup Q}(U_Q) \otimes_\cO \cO/\lambda^N)_{\ffrm_Q}[I]$ such that $\lambda^{N-m-1} e \neq 0$ and $\Frob_v^{-1} e = q_v^{-1} f(U_v) e$. The same argument as in the case $\# Q$ even now shows that the image of $e$ in $(H_{R \cup \overline{Q}}(U_{\overline{Q}}) \otimes_\cO \cO/\lambda^N)_{\ffrm_Q}^2[I]$ satisfies $\lambda^{N - 2m - 1} e \neq 0$, and leads to the desired homomorphism $f'$.
\end{proof}

\begin{lemma}\label{types}
Let $\pi$ be a cuspidal automorphic representation  of $\GL_2(\A_F)$ of weight 2. Suppose that for each finite place $v \notin S_p$ of $F$, either $\pi_v$ is unramified or $q_v \equiv 1 \text{ mod }p$ and $\pi_v$ is an unramified twist of the Steinberg representation, while for every $v \in S_p$, $\pi_v$ is $\iota$-ordinary and $\pi_v^{U_0(v)} \neq 0$. Suppose furthermore  that $\overline{r_\iota(\pi)}|_{G_{F(\zeta_p)}}$ is irreducible. Let $\sigma \subset S_p$ be a subset. Then there exists a cuspidal automorphic representation $\pi'$ of $\GL_2(\A_F)$ of weight 2 satisfying the following conditions:
 \begin{enumerate}
\item There is an isomorphism of residual representations $\overline{r_\iota(\pi)}= \overline{r_\iota(\pi')}$, and $\pi,\pi'$ have the same central character.
\item If $v \in \sigma$, then $\pi'_v$ is $\iota$-ordinary. If $v \in S_p -\sigma$, then $\pi'_v$ is supercuspidal.
\item If $v$ is place of $F$ that  does not divide  $ p\infty$  and $\pi_v$ is unramified, then $\pi'_v$ is unramified. If $\pi_v$ is  ramified, then $\pi'_v$ is ramified principal series representation. 
\end{enumerate}
\end{lemma}
\begin{proof}
The lemma is the same as \cite[Lemma 5.25]{ThorneA}, except that the condition $\overline{r_\iota(\pi)}$ irreducible and $[F(\zeta_p) : F] \geq 4$ has been replaced by the condition $\overline{r_\iota(\pi)}|_{G_{F(\zeta_p)}}$ irreducible. This is justified by Lemma \ref{trivial}.
\end{proof}

\section{Galois theory}\label{sec_galois_theory}

In this paper, we use the notation from Galois deformation theory introduced in \cite[\S 5]{ThorneA}. We first recall the basic definitions and then study the problem of killing dual Selmer for an exceptional residual representation. 

\subsection{Galois deformation theory}\label{sec_deformation_theory}

Let $p$ be an odd prime, let $F$ be a number field, and let $E$ be a coefficient field. We fix a continuous, absolutely irreducible representation $\overline{\rho} : G_F \to \GL_2(k)$ and a continuous character $\mu : G_F \to \cO^\times$ which lifts $\det \overline{\rho}$. We will assume that $k$ contains the eigenvalues of all elements in the image of $\overline{\rho}$. We also fix a finite set $S$ of finite places of $F$, containing the set $S_p$ of places dividing $p$, and the places at which $\overline{\rho}$ and $\mu$ are ramified. For each $v \in S$, we fix a ring $\Lambda_v \in \CNL_\cO$ and we define $\Lambda = \widehat{\otimes}_{v \in S} \Lambda_v$, the completed tensor product being over $\cO$. Then $\Lambda \in \CNL_\cO$.

Let $v \in S$. We write $\cD_v^\square : \CNL_{\Lambda_v} \to \mathrm{Sets}$ for the functor that associates to $R \in \mathrm{CNL}_{\Lambda_v}$ the set of all continuous homomorphisms $r : G_{F_v} \to \GL_2(R)$ such that $r \mod \ffrm_R = \overline{\rho}|_{G_{F_v}}$ and $\det r$ agrees with the composite $G_{F_v} \to \cO^\times \to R^\times$ given by $\mu|_{G_{F_v}}$ and the structural homomorphism $\cO \to R$. It is easy to see that $\cD_v^\square$ is represented by an object $R_v^{\square} \in \CNL_{\Lambda_v}$.
\begin{definition}
Let $v \in S$. A local deformation problem for $\overline{\rho}|_{G_{F_v}}$ is a subfunctor $\cD_v \subset \cD_v^\square$ satisfying the following conditions:
\begin{itemize}
\item $\cD_v$ is represented by a quotient $R_v$ of $R_v^\square$.
\item For all $R \in \CNL_{\Lambda_v}$, $a \in \ker(\GL_2(R) \to \GL_2(k))$ and $r \in \cD_v(R)$, we have $a r a^{-1} \in \cD_v(R)$.
\end{itemize}
\end{definition}
We will write $\rho_v^\square : G_{F_v} \to \GL_2(R_v^\square)$ for the universal lifting. If a quotient $R_v$ of $R_v^\square$ corresponding to a local deformation problem $\cD_v$ has been fixed, we will write $\rho_v : G_{F_v} \to \GL_2(R_v)$ for the universal lifting of type $\cD_v$.
\begin{definition}
A global deformation problem is a tuple
\[ \cS = (\overline{\rho}, \mu, S, \{ \Lambda_v \}_{v \in S}, \{ \cD_v \}_{v \in S}), \]
where:
\begin{itemize}
\item The objects $\overline{\rho} : G_F \to \GL_2(k)$, $\mu : G_F \to k^\times$, $S$ and $\{ \Lambda_v \}_{v \in S}$ are as at the beginning of this section.
\item For each $v \in S$, $\cD_v$ is a local deformation problem for $\overline{\rho}|_{G_{F_v}}$.
\end{itemize}
\end{definition}
\begin{definition}
Let $\cS = (\overline{\rho}, \mu, S, \{ \Lambda_v \}_{v \in S}, \{ \cD_v \}_{v \in S})$ be a global deformation problem. Let $R \in \CNL_{\Lambda}$, and let $\rho : G_F \to \GL_2(R)$ be a continuous lifting of $\overline{\rho}$. We say that $\rho$ is of type $\cS$ if it satisfies the following conditions:
\begin{itemize}
\item $\rho$ is unramified outside $S$.
\item $\det \rho = \mu$. More precisely, the homomorphism $\det \rho : G_F \to R^\times$ agrees with the composite $G_F \to \cO^\times \to R^\times$ induced by $\mu$ and the structural homomorphism $\cO \to R$.
\item For each $v \in S$, the restriction $\rho|_{G_{F_v}}$ lies in $\cD_v(R)$, where we give $R$ the natural $\Lambda_v$-algebra structure arising from the homomorphism $\Lambda_v \to \Lambda$.
\end{itemize}
We say that two liftings $\rho_1, \rho_2 : G_F \to \GL_2(R)$ are strictly equivalent if there exists a matrix $a \in \ker(\GL_2(R) \to \GL_2(k))$ such that $\rho_2 = a \rho_1 a^{-1}$. 
\end{definition}
It is easy to see that strict equivalence preserves the property of being of type $\cS$. We write $\cD^\square_\cS$ for the functor $\CNL_{\Lambda} \to \Sets$ which associates to $R \in \CNL_{\Lambda}$ the set of liftings $\rho : G_F \to \GL_2(R)$ which are of type $\cS$. We write $\cD_\cS$ for the functor $\CNL_{\Lambda} \to \Sets$ which associates to $R \in \CNL_{\Lambda}$ the set of strict equivalence classes of liftings of type $\cS$. 

\begin{definition} If $T \subset S$ and $R \in \mathrm{CNL}_\Lambda$, then we define a $T$-framed lifting of $\overline{\rho}$ to $R$ to be a tuple $(\rho, \{ \alpha_v \}_{v \in T})$, where $\rho : G_F \to \GL_2(R)$ is a lifting and for each $ v\in T$, $\alpha_v$ is an element of $\ker(\GL_2(R) \to \GL_2(k))$. Two $T$-framed liftings $(\rho_1, \{ \alpha_v \}_{v \in T})$ and $(\rho_2, \{ \beta_v \}_{v \in T})$ are said to be strictly equivalent if there is an element $a \in \ker(\GL_2(R) \to \GL_2(k))$ such that $\rho_2 = a \rho_1 a^{-1}$ and $\beta_v = a \alpha_v$ for each $v \in T$.
\end{definition}
We write $\cD^T_\cS$ for the functor $\CNL_{\Lambda} \to \Sets$ which associates to $R \in \CNL_{\Lambda}$ the set of strict equivalence classes of $T$-framed liftings $(\rho, \{ \alpha_v \}_{v \in T})$ to $R$ such that $\rho$ is of type $\cS$. The functors $\cD_\cS$, $\cD_\cS^\square$ and $\cD_\cS^T$ are represented by objects $R_\cS$, $R_\cS^\square$ and $R_\cS^T$, respectively, of $\CNL_{\Lambda}$. The cohomology groups $H^\ast_{\cS, T}(\Ad^0 \rhobar)$ and the dual Selmer group $H^1_{\cS, T}(\Ad^0 \rhobar(1))$ are defined just as in \cite[\S 5.2]{ThorneA}. The cohomological machinery developed in \cite[\S 5.2]{ThorneA} then implies:
\begin{lemma}
Assume $n=2$, and suppose further that $F$ is totally real, $\rhobar$ is totally odd and $\overline{\rho}|_{G_{F(\zeta_p)}}$ is absolutely irreducible,  and $R_v$ is formally smooth over $\Oc$ of dimension $4$ for each $v \in  S - T$ . Then $R_{\cS}^T$ is a quotient of a power series ring over $A_{\cS}^T = \widehat{\otimes}_{v \in T} R_v$ in $h^1_{\cS,T} (\Ad^0 \rhobar(1)) - [F:\Q] - 1 + \#T $ variables.
\end{lemma}
\begin{proof}
The proof is the same as \cite[Corollary 5.7]{ThorneA}, except that the assumption that $\overline{\rho}|_{G_{F(\zeta_p)}}$ is absolutely irreducible replaces the condition $[F(\zeta_p) : F] > 2$.
\end{proof}
\subsection{Another local deformation problem}\label{sec_another_local_deformation_problem}

To the glossary of deformation problems considered in \cite[\S 5.3]{ThorneA}, we add the following deformation problem that is important to the work of this paper. Let $v \in S - S_p$, and suppose that $q_v \equiv 1 \text{ mod }p$, $\overline{\rho}|_{G_{F_v}}$ is unramified, and $\rhobar(\Frob_v)$ is unipotent of order $p$. We define a subfunctor  $\cD_v^{\rm St(uni)} \subset \cD_v^\square$ by declaring that for $R \in \CNL_\Oc$, $\cD_v^{\rm St(uni)}(R)$ is the set of lifts of $\rhobar|_{G_{F_v}}$ to $\GL_2(R)$ of fixed determinant $\mu$  which are of the form $$\left(\begin{array}{cc}
  \chi  & *  \\
  0 & \chi\epsilon^{-1}
\end{array} \right),$$ with $\chi$ an unramified character.

\begin{prop}\label{steinberg}
The functor $\cD_v^{\rm St(uni)}$ is a local deformation problem. The representing object $R_v^{\rm St(uni)}$ is formally smooth over $\Oc$ of dimension 4. Further $\cL_v$ and $\cL_v^\perp$  are 1-dimensional, and transverse to the 1-dimensional unramified subspaces of $H^1(F_v,\Ad^0 \rhobar)$ and $H^1(F_v,\Ad^0\rhobar (1))$ respectively. If $\overline{\rho}$ is exceptional, then the image of the Lie subspace of $H^1(F_S/F,\Ad^0 \rhobar(1))$ in $H^1(F_v, \Ad^0 \rhobar(1))$ is not in $\mathcal L_v^\perp$.
\end{prop}

\begin{proof}
The key point as follows. Consider  matrices in $A,B \in \GL_2(R)$ such that $A=\left( \begin{array}{cc}
  \alpha & *  \\
  0 & \beta
\end{array} \right)$ reduces to a unipotent matrix of order $p$ and $B$ reduces to the identity. If   
\[ B{\left( \begin{array}{cc}
  \alpha & *  \\
  0 & \beta
\end{array} \right)}B^{-1}=\left( \begin{array}{cc}
  \alpha & *'  \\
  0 & \beta
\end{array} \right), \] then $B$ is upper triangular (cf.  E3 of \cite[\S 1]{Tay}). The claim about $\cL_v$ and $\cL_v^\perp$ follows easily.
\end{proof}

\subsection{Auxiliary primes}\label{novel}

Let us now suppose that $F$ is totally real and $\overline{\rho}$ totally odd, and let $F_0 \subset \overline{F}$ denote the the fixed field of $\ad^0 \overline{\rho}$. Consider the inflation-restriction exact sequence 
\[ 0 \rightarrow H^1({\rm Gal}(F_0/F),\Ad^0 \rhobar (1)^{G_{F_0}}) \rightarrow H^1(G_F,\Ad^0 \rhobar (1)) \rightarrow H^1(G_{F_0},\Ad^0 \rhobar (1))^{{\rm Gal}(F_0/F)}.\]
  \begin{lemma}
  The $k$-vector space $ H^1({\rm Gal}(F_0/F),\Ad^0(1)^{G_{F_0}})$  is non-zero precisely when $p=5$,  the projective image
 of $\rho$ is $\PGL_2(\F_5)$ and $\zeta_5 \in F$. (This implies that $\sqrt{5} \in F$, equivalently  $[F(\zeta_5):F]=2$.) In this  exceptional case it is one dimensional.
 \end{lemma}
 \begin{proof}
 This follows from \cite[Lemma 2.48]{DDT} and its proof. 
  \end{proof}
We now specialize to the case when the conclusion of the lemma is satisfied. Thus we assume $p = 5$, $\Gal(F_0 / F) \cong \PGL_2(\bbF_5)$, and $\zeta_5 \in F_0$. For $m \geq 1$, define $F_m=F_0(\zeta_{5^m})$. (Thus $F_0=F_1$.)  Suppose given as well a lift $\rho : G_F \to \GL_2(\cO)$ of $\overline{\rho}$, and set for each $N \geq 1$ $\rho_N = \rho \text{ mod }\lambda^N$. 

\begin{lemma}\label{special1}
For all  $N \gg 0$, there is a set  of places $v $ of $F$ of positive density such that  $\rho_N(\Frob_v)=\left( \begin{array}{cc}
  \alpha & *  \\
  0 & \beta
\end{array} \right)$   reduces to a unipotent matrix of order $p=5$  and such that $\alpha\beta^{-1}=q_v$, with $q_v \equiv 1 \text{ mod }5$, $q_v \not\equiv 1 \text{ mod }5^2$. 
\end{lemma}
\begin{proof}
It follows from Lemma \ref{Z5} that the field extensions $K_1,K_2$ of $F$  cut out by 
 the projective image of $\rho$ and the $p$-adic cyclotomic  character are almost linearly disjoint
 over $F$, i.e. $K_1 \cap K_2$ is a finite abelian  extension of $F$, and in fact $K_1 \cap K_2=F(\zeta_5)$.

Let $H$ denote the projective image of $\rho$, and let $H_1$ be  the subgroup of $H$ fixing $K_1 \cap K_2$. By Lemma \ref{Z5}, $H$ contains a conjugate of $\PGL_2(\Z_5)$, hence $H_1$ contains a  conjugate of  $\PSL_2(\Z_5)$. The result  now follows  from  the Cebotarev density theorem. (Note that if $\overline{\rho}(\Frob_v) \in \PSL_2(\bbF_5)$ then $q_v \equiv   1 \text{ mod }5$, because $\overline{\rho}$ is exceptional.)
\end{proof}

\begin{prop}\label{auxiliary}
  Let $\psi$ be a non-zero element in the 1-dimensional $k$-vector space $H^1(\Gal(F_0/F),\Ad^0\rhobar(1)^{G_{F_0}})$. For each integer $N \geq 1$ there is a set of places $v$ of $F$ of positive density such that:
  \begin{enumerate}
   \item We have $q_v \equiv 1 \text{ mod }5$ but $q_v \not\equiv 1\text{ mod }5^2$.
  \item $\rhobar$ is unramified at $v$.
  \item $\rhobar({\rm Frob}_v)$ is  unipotent of order $5$, and $\psi|_{G_{F_v}}$ is unramified at $v$ and non-zero.
  \item The restriction of $\rho_{N}$ to $G_{F_v}$  is of the form
  \[ \left(\begin{array}{cc}
  \chi  & *  \\
  0 & \chi\epsilon^{-1}
\end{array} \right).\]
  \end{enumerate}
\end{prop}
\begin{proof}
 This follows from  Lemma \ref{special1}. 
\end{proof}
We now show that when a prime of type $\cD_v^{\St(\text{uni})}$ is present, we can use Taylor--Wiles primes to kill the remainder of the dual Selmer group.
\begin{prop}\label{TW}
Let $$\cS=(\rhobar,\mu,S,\{\Lambda_v\}_{v \in S},\{\cD_v\}_{v \in S} )$$ be a global deformation problem. Let $T \subset S$ and suppose that for each $v \in S-T$, $\cD_v$ is  
$\cD_v^{St(\alpha_v)}$ or $\cD_v^{\rm St(uni)}$ and there is at least one place of the latter type  in $S$. Then for each $N_1 \geq 1$, there exists a finite set of primes $Q_1$, disjoint from $S$, satisfying the following conditions:
\begin{enumerate} \item $\#Q_1=h^1_{\cS,T}(\Ad^0(1))$ and for each $v \in Q_1$, $q_v \equiv 1 \text{ mod }p^{N_1}$, and $\rhobar(\Frob_v)$ has distinct eigenvalues.
\item Define the augmented deformation  problem: 
\[ \cS_{Q_1}=(\rhobar,\mu,S \cup Q_1,\{\Lambda_v\}_{v \in S} \cup \{ \cO \}_{v \in Q_1},\{\cD_v\}_{v \in S} \cup \{\cD_v^\square\}_{v \in Q_1}).\]
  Then $h^1_{\cS_{Q_1},T}(\ad^0\rhobar(1))=0$.
\end{enumerate}
\end{prop}
\begin{proof}
By Proposition \ref{steinberg}, the existence of a place $v \in S$ with deformation problem $\cD_v^{\rm St(uni)}$ implies that the Lie subspace of  $H^1_{\cS,T}(S, \Ad^0 \rhobar(1))$ is zero.  Further one checks that $$H^1({\rm Gal}(F_m/F_0),\Ad^0\rhobar(1))^{G_F}= \Hom(\Gal(F_m/F_1),\Ad^0\rhobar(1)^{G_F})=0$$ using that $\rhobar|_{G_{F(\zeta_p})}$  is irreducible. Thus   $H^1_{\cS,T}(S, {\rm Ad}^0\rhobar(1))$  maps injectively under the restriction map to $H^1(F_m,\ad^0\rhobar(1))$, and then by the usual arguments (see \cite[Theorem 2.49]{DDT} and \cite[Proposition 5.20]{ThorneA})  one can find a Taylor--Wiles set  $Q_1$ with the required properties.
\end{proof}
\section{$R=\T$}\label{sec_r_equals_t}
Let $p = 5$, let $E$ be a coefficient field, and let $F$ a totally real number field of even degree over $\Q$. We fix an absolutely irreducible, totally odd  and continuous representation $\rhobar:G_F \ra \GL_2(k)$ satisfying the following conditions.
\begin{itemize}

\item $\rhobar$ is exceptional.

\item For each place $ v$ of $F$ prime to $p$, $\rhobar|_{F_v}$ is unramified.

\item For each place $v|p$ of $F$, $\rhobar|_{G_{F_v}}$ is trivial.

\item The character $\epsilon\det \rhobar$ is everywhere unramified.

\end{itemize}

We assume $k$ large enough that it contains the eigenvalues of every element of the image of $\rhobar$. We write $\psi:G_F \ra \Oc^*$ for the Teichm\"uller lift of  $\epsilon \det \rhobar$. In what follows we will abuse notation by also writing $\psi$ for the character $\psi \circ \Art_F:\A_F^{\times} \ra \Oc^\times$. We will also suppose given the following data:

\begin{itemize}
\item A finite set $R$ of even cardinality of finite places of $F$, such that for $v \in R$, $q_v \equiv 1 \text{ mod }p$,  and $\rhobar|_{G_{F_v}}$ is trivial.

\item  A finite set $Q_0$ of finite places of even cardinality, disjoint from $S_p \cup R$, which can be decomposed as $V_1 \cup V_2$,  and a tuple $(\alpha_v)_{v \in V_1}$ of elements in $k$, such that for $v \in V_1$ the deformation problem $\cD_v^{\rm St(\alpha_v)}$ is defined and for $v \in V_2$ the deformation problem $\cD_v^{\rm St(uni)}$ is defined. In particular, $q_v \equiv -1 \text{ mod }p$ if $v \in V_1$ and $q_v \equiv 1 \text{ mod }p$ if $v \in V_2$.

\item An isomorphism $\iota: \overline{\Q}_p \cong \C$, and a cuspidal automorphic representation $\pi_0$ of weight $2$ satisfying the following conditions:
\begin{itemize} \item There is an isomorphism $\overline{r_\iota(\pi)} \cong \rhobar.$
\item The central character of $\pi_0$ equals $\psi$.
\item For each finite place $v \notin S_p \cup R \cup Q_0$ of $F$, $\pi_{0,v}$ is unramified.
\item For each $v \in R \cup Q_0$, there is an unramified character $\chi_v : F_v^\times \ra \overline{\Q}_p^\times$ and an isomorphism  $\pi_{0,v} \cong {\rm St}(\iota\chi_v)$, For each $v \in V_1$, $\chi_0(\varpi_v)$ is congruent to $\alpha_v$ modulo the maximal ideal of $\overline{\Z}_p$.
\item Let $\sigma \subset S_p$ denote the set of places $v$ such that $\pi_{0,v}$ is $\iota$-ordinary. For each $v \in \sigma$, $\pi_{0,v}^{U_0(v)} \neq 0$. For each $v \in S_p -\sigma$, $\pi_{0,v}$ is unramified. 
\end{itemize}
\end{itemize}
We consider the global deformation problem 
\[\cS=(\rhobar,\epsilon^{-1}\psi,S_p \cup Q_0 \cup R, \{\Lambda_v\}_{v \in \sigma} \cup \{\Oc\}_{v \in (S_p-\sigma) \cup Q_0\cup R},$$$$ \{\cD_v^{\text{ord}}\}_{v \in \sigma} \cup \{\cD_v^{\text{non-ord}}\}_{v\in S_p -\sigma} \cup \{\cD_v^{\rm St(\alpha_v)}\}_{v \in V_1} \cup \{\cD_v^{\rm St (uni)}\}_{v \in V_2}\cup \{\cD_v^{\rm St}\}_{v \in R}).\]
The local deformation problems are as defined in \cite[\S 5.3]{ThorneA} and \S \ref{sec_another_local_deformation_problem} of the current paper. The rings $\Lambda_v$ for $v \in \sigma$ are the local Iwasawa algebras, as in \cite[\S 6]{ThorneA}. We fix an auxiliary place $a$ of $F$ as in the following lemma:
\begin{lemma}\label{trivial}
There exists a place $a \notin S_p \cup Q_0 \cup R$ such that $q_a >4^{[F:\Q]}$, $\tr \rhobar(\Frob_a)/\det \rhobar(\Frob_a) \neq (1+q_a)^2/q_a$, and $q_a \not\equiv 1 \text{ mod }p$.
\end{lemma}
\begin{proof}
See \cite[Lemma 12.3]{Jar99} and the remarks immediately afterwards. We only need that $\overline{\rho}|_{G_{F(\zeta_p)}}$ is irreducible (which is certainly the case if $\overline{\rho}$ is exceptional).
\end{proof}
 We set $\Lambda = \widehat{\otimes}_{v \in \sigma} \Lambda_v$, and write $\Lambda_n$ for its `level $n$' quotient, as in \cite[\S 6.2]{ThorneA}. The situation is now exactly as in \cite[\S 6]{ThorneA}. We can take the quaternion algebra $B$ ramified at $Q_0 \cup R \cup \{ v | \infty \}$, and an appropriate level subgroup $U \subset (B \otimes_F \bbA_F^{\infty})^\times$ as in \cite[\S 6.1]{ThorneA} (with $U^1_1(a)$-level structure at the place $a$). If $n \geq 1$, then we can define the spaces $H_\psi^\text{ord}(U_1(\sigma^n), \cO)$ and $H_\psi^\text{ord}(U_1(\sigma^n), k)$ of $\sigma$-ordinary modular forms on $B$ of level $U_1(\sigma^n)$, weight 2, and central character $\psi$. The existence of $\pi_0$, together with the Jacquet--Langlands correspondence, implies the existence of a maximal ideal ${\mathfrak m}$ of the Hecke algebra $\bbT^{\Lambda, S_p \cup Q_0 \cup R \cup \{ a \}}_{Q_0}(H_\psi(U_1(\sigma^n), \cO))$ such that $\overline{\rho}_{\mathfrak m} \cong \overline{\rho}$, and we then obtain a surjection $R_{\cS} \to \bbT^{\Lambda, S_p \cup Q_0 \cup R \cup \{ a \}}_{Q_0}(H_\psi(U_1(\sigma^n), \cO))_\ffrm$ (cf. \cite[Proposition 6.5]{ThorneA}).  

Having introduced this notation, we can state the following result, which is the analogue in our situtation of \cite[Corollary 6.7]{ThorneA}:
\begin{theorem}\label{cor_r_equals_t}
Let $C, N, n \geq 1$ be integers, and suppose that $V_2 \neq \emptyset$. Suppose that $\dim_k H_\psi^\text{ord}(U_1(\sigma^n), k)[\ffrm] \leq C$, and that there is a diagram
\[ \xymatrix{ \Lambda \ar[r] \ar[d] & \Lambda_n \ar[d] \\
R_\cS \ar[r] & \cO/\lambda^N, }\]
corresponding to a lifting $\rho_N : G_F \to \GL_2(\cO/\lambda^N)$ of $\overline{\rho}$ which is of type $\cS$. Let  $I = \ker ( R_\cS \to \cO/\lambda^{\lfloor N/C \rfloor} )$. Then $(H_\psi(U_1(\sigma^n), \cO)_\ffrm \otimes_\cO \cO/\lambda^{\lfloor N/C \rfloor})[I]$ contains an $\cO$-submodule isomorphic to $\cO/\lambda^{\lfloor N/C \rfloor}$, and the map $R_\cS \to \cO/\lambda^{\lfloor N/C \rfloor}$ factors 
\[ R_\cS \twoheadrightarrow \bbT^{\Lambda, S_p \cup Q_0 \cup R \cup \{ a \}}_{Q_0}(H^\text{ord}_\psi(U_1(\sigma^n), \cO)_\ffrm) \to \cO/\lambda^{\lfloor N/C \rfloor}. \]
\end{theorem}
\begin{proof}
The proof of Theorem \ref{cor_r_equals_t} is exactly the same as the proof of \cite[Corollary 6.7]{ThorneA}, except that references to \cite[Lemma 6.9]{ThorneA} should be replaced by references to Lemma \ref{lem_new_tw_prime} below.
\end{proof}

\begin{lemma}\label{lem_new_tw_prime}
Let $q=h^1_{\cS,T}(\Ad^0\rhobar(1))$, and suppose that $V_2 \neq \emptyset$. Then for all integers $N \geq 1$, there exists a set $Q_N$ of places of $F$ satisfying the following conditions.
\begin{enumerate}
\item $Q_N \cap (S_p \cup Q_0 \cup R \cup \{a\})=\emptyset$ and $\#Q=q$.
\item For each $v \in Q_N$, $q_v \equiv 1 \text{ mod }p^N$.
\item For each $v \in Q_N$, $\rhobar(\Frob_v)$ has distinct eigenvalues $\alpha_v,\beta_v$ in $k$.
\item Consider the augmented deformation problem
\[\cS_{Q_N} =(\rhobar,\epsilon^{-1}\psi,S_p \cup Q_0 \cup R \cup Q_N, \{\Lambda_v\}_{v \in \sigma} \cup \{\Oc\}_{v \in (S_p-\sigma) \cup Q_0\cup R \cup Q_N},\]
\[ \{\cD_v^{\text{ord}}\}_{v \in \sigma} \cup \{\cD_v^{\text{non-ord}}\}_{v\in S_p -\sigma} \cup \{\cD_v^{\rm St(\alpha_v)}\}_{v \in V_1} \cup \{\cD_v^{\rm St (uni)}\}_{v \in V_2}\cup \{\cD_v^{\rm St}\}_{v \in R} \cup \{ \cD_v^\square \}_{v \in Q_N}).\]
Then $h^1_{S_{Q_N},T}(\Ad^0\rhobar(1))=0$.
\item The ring $R_{\cS_{Q_N}}^T$ can be written as a quotient of a power series over $A_{\cS_{Q_N}}^T=A_{\cS}^T$ in $q-[F:\Q]-1+\#T$ variables.
\end{enumerate}
\end{lemma}
\begin{proof}
This follows immediately from \cite[Proposition 5.6]{ThorneA} and Proposition \ref{auxiliary}.
\end{proof}

\section{Ordinary automorphic liftings of exceptional $\rhobar$}\label{ordinary}
 Let $p = 5$, and fix throughout this section a choice of isomorphism $\iota : \overline{\bbQ}_5 \cong \bbC$. In this section, which can be read independently from the rest of this paper, we prove the following theorem. The role it plays in this paper is analogous to the role of \cite[Lemma 7.3]{ThorneA} in \emph{op. cit.}.
\begin{theorem}\label{thm_existence_of_ordinary_lift_of_exceptional_representation}
Let $F$ be a totally real field. Suppose that $\overline{\rho} : G_F \to \GL_2(\overline{\bbF}_5)$ is an automorphic and exceptional residual representation. Then there exists a soluble totally real extension $F'/F$ and a cuspidal automorphic representation $\pi$ of $\GL_2(\bbA_{F'})$ of weight 2 and satisfying the following conditions:
\begin{enumerate}
\item $\pi$ is everywhere unramified and $\iota$-ordinary.
\item There is an isomorphism $\overline{r_\iota(\pi)} \cong \overline{\rho}|_{G_{F'}}$.
\end{enumerate}
\end{theorem}
The key point is that the automorphic representation $\pi$ can be chosen to be $\iota$-ordinary. To prove Theorem \ref{thm_existence_of_ordinary_lift_of_exceptional_representation}, we will follow the strategy of \cite{Bar12}, namely taking a tensor product of $\overline{\rho}$ with another 2-dimensional representation and using an automorphy lifting theorem in 4 dimensions. The main difference compared to \emph{op. cit.} is that we are working here in the exceptional (viz. `non-adequate') case. We therefore require a 4-dimensional automorphy lifting result that works in this situation, and some extra work is required to establish this, which parallels what we do for $\GL_2$ in the rest of this paper. We first prove this automorphy lifting result in \S \ref{sec_automorphy_in_4_dimensions}. We then prove a version of Theorem \ref{thm_existence_of_ordinary_lift_of_exceptional_representation} under supplementary hypotheses, in \S \ref{sec_ordinary_lifts_under_additional_conditions}. We deduce the general version of the theorem in \S \ref{sec_ordinary_lifts_in_general}.

\subsection{An automorphy lifting theorem in 4 dimensions}\label{sec_automorphy_in_4_dimensions}
We adopt the deformation-theoretic notation of the papers \cite{Clo08} and \cite{ThorneB}. In particular, we make use of the group $\cG_n$, defined on \cite[p. 8]{Clo08}, which is used to describe the deformation theory of conjugate self-dual Galois representations.

Let $L$ be a coefficient field with residue field $k$. Let $E/F$ be an everywhere unramified quadratic imaginary extension of a totally real field, and let $c \in \Gal(E/F)$ denote the non-trivial element. We suppose that each $p$-adic place $v$ of $F$ splits in $E$, and write $\wv$ for a choice of place of $E$ above $v$. Suppose given continuous representations $\overline{\rho}, \overline{\sigma} : G_E \to \GL_2(k)$ satisfying the following conditions:
\begin{enumerate}
\item The tensor product $\overline{\tau} = \overline{\rho} \otimes \overline{\sigma}$ is absolutely irreducible and satisfies $\overline{\tau}^c \cong \overline{\tau}^\vee \epsilon^{-3}$. The extensions of $E$ cut out by $\overline{\rho}$ and $\overline{\sigma}$ are linearly disjoint. The field $k$ contains the eigenvalues of all elements in the image of $\overline{\tau}$.
\item The representation $\overline{\rho}$ has projective image $\PGL_2(\bbF_5)$, and the quadratic extension cut out by the quotient $\PGL_2(\bbF_5) / \PSL_2(\bbF_5)$ is $E(\zeta_5)$. (In particular, $\zeta_5 \not\in E$.)
\item The representation $\overline{\sigma}$ has image of order prime to $p$.
\item There exist places $w_1, \dots, w_s$ of $E$, split over $F$, and satisfying the following conditions:
\begin{enumerate}
\item For each $i = 1, \dots, s$, we have $q_{w_i} \equiv 1 \text{ mod }p$.
\item For each $i = 1, \dots, s$, the representation $\overline{\rho}|_{G_{E_{w_i}}}$ is unramified and $\overline{\rho}(\Frob_{w_i})$ is unipotent of order $p$. The representations $\overline{\sigma}|_{G_{E_{w_i}}}$ are absolutely irreducible.
\item Writing $v_i$ for the places of $F$ below $w_i$, we have that the places $v_i$ are pairwise distinct and $s$ is \emph{even}.
\end{enumerate}
\item The representations $\overline{\rho}$ and $\overline{\sigma}$ are unramified at all finite places $w \nmid v_1 \dots v_s$.
\item There exists a RACSDC automorphic representation $\Pi$ of $\GL_4(\bbA_E)$ satisfying the following conditions:
\begin{enumerate}
\item For each finite place $w\nmid v_1 \dots v_s$ of $E$, $\Pi_w$ is unramified. For each $i = 1, \dots, s$, there exists a supercuspidal representation $\mu_i$ of $\GL_2(E_{w_i})$ and an isomorphism $\Pi_{w_i} \cong \text{Sp}_2(\mu_i)$, where $\text{Sp}_n$ is as defined on \cite[p. 32]{Tay01}. If $\omega_{\mu_i}$ denotes the central character of $\mu_i$, then $\mu_i(\cO_{E_w}^\times)$ has order prime to $p$.
\item $\Pi$ has weight 0 (in the sense of \cite[\S 4.2]{Clo08}) and there is an isomorphism $\overline{r_\iota(\Pi)} \cong \overline{\tau}$.
\end{enumerate}
\end{enumerate}
We fix an extension of $\overline{\tau}$ to a continuous homomorphism $\overline{\tau} : G_F \to \cG_4(k)$ with $\nu \circ \overline{\tau} = \epsilon^{-3}$ (which exists, by \cite[Lemma 2.1.4]{Clo08}) and define two global deformation problems (as on \cite[p. 27]{Clo08}) in 4 dimensions relative to the representation $\overline{\tau}$:
\begin{equation}\label{eqn_deformation_problem_s} \cS = (E/F, S, \widetilde{S}, \cO, \overline{\tau}, \epsilon^{-3}, \{ \cD_v \}_{v \in S}), 
\end{equation}
\begin{equation}\label{eqn_deformation_problem_sprime} \cS' = (E/F, S, \widetilde{S}, \cO, \overline{\tau}, \epsilon^{-3}, \{ \cD'_v \}_{v \in S}). 
\end{equation}
Here $S = \{ v | p \} \cup \{ v_1, \dots, v_s \}$, and $\widetilde{S} = \{ \wv | p \} \cup \{ w_1, \dots, w_s \}$. For each $v \in S_p$, $\cD_v$ is the local deformation problem defined by the weight 0 crystalline lifting ring $R^\text{crys}$ of $\overline{\tau}|_{G_{E_\wv}}$ (which is non-zero because of the existence of $r_\iota(\Pi)|_{G_{E_\wv}}$), while $\cD'_v$ is the local deformation problem defined by the (unique) irreducible component of $\Spec R^\text{crys}$ containing the point corresponding to $r_\iota(\Pi)|_{G_{E_\wv}}$. (For the existence and properties of these deformation rings, see \cite[p. 863]{ThorneB}.) For $v \in \{ v_1, \dots, v_s \}$, $\cD_v = \cD'_v$ is the local deformation problem whose existence is asserted by Proposition \ref{prop_steinberg_tensor_cuspidal_liftings}:
\begin{proposition}\label{prop_steinberg_tensor_cuspidal_liftings}
Let $i \in \{ 1, \dots, s \}$ and write $v = v_i$, $w = w_i$. Let $\sigma : G_{E_{w}} \to \GL_2(\cO)$ denote a lifting of $\overline{\sigma}|_{G_{E_{w}}}$ to $\GL_2(\cO)$ such that $\sigma(I_{E_w})$ has order prime to $p$. It exists because the representation $\overline{\sigma}|_{G_{E_{w}}}$ is absolutely irreducible and has image of order prime to $p$. If $R \in \CNL_\cO$, then we write $\cD_v(R)$ for the functor of liftings of $\overline{\tau}|_{G_{E_w}}$ conjugate to one of the form $\rho \otimes \sigma$, where $\rho : G_{E_w} \to \GL_2(R)$ is a unipotently ramified lifting of $\overline{\rho}|_{G_{E_w}}$ such that for any Frobenius lift $\phi$, we have $(\tr \rho(\phi))^2 = (q+1)^2/q \det \rho(\phi)$.

Then $\cD_v$ is indeed a local deformation problem, and is formally smooth over $\cO$.
\end{proposition}
\begin{proof}
It follows from \cite[Corollary 2.4.13]{Clo08} that deforming $\overline{\tau}|_{G_{E_w}}$ is indeed equivalent to deforming $\overline{\rho}|_{G_{E_w}}$. The proof of the proposition therefore reduces to the proof of Proposition \ref{steinberg}. We are using here the fact that $\overline{\rho}(\Frob_w)$ has order $p$ (i.e.\ is unipotent and non-trivial). 
\end{proof}
The universal deformation rings $R_\cS^\text{univ}, R_{\cS'}^\text{univ} \in \CNL_\cO$ are defined, and there is a natural surjective homomorphism $R_{\cS}^\text{univ} \to R_{\cS'}^\text{univ}$ (corresponding to the fact that for each $v \in S$, $\cD'_v$ is a subfunctor of $\cD_v$).
\begin{theorem}\label{thm_automorphy_lifting_in_dimension_4}
The $\cO$-algebra $R_{\cS'}^\text{univ}$ is a finite $\cO$-module. If $\tau : G_F \to \cG_n(\overline{\bbQ}_p)$ is a representation corresponding to a homomorphism $R_{\cS'}^\text{univ} \to \overline{\bbQ}_p$, then there exists a RACSDC automorphic representation $\Pi'$ of $\GL_4(\bbA_E)$ such that $\tau|_{G_E} \cong r_\iota(\Pi')$. 
\end{theorem}
\begin{proof}
The proof is an application of the Taylor--Wiles argument, as found for example in \cite{ThorneB} and especially the proof of \cite[Theorem 6.8]{ThorneB}. There are two modifications that we must account for: first, we must carefully choose a unitary group and level structure so that the relevant space of automorphic forms gives rise to Galois representations of type $\cS$. Second, we must construct sufficiently many sets of Taylor--Wiles primes under a weakening of the `adequacy' hypothesis appearing in \emph{op. cit.}. We now describe the proof, omitting details when no novel argument is required. 

In defining our unitary groups and automorphic forms, we follow \cite[\S 3.3]{Clo08}. In the notation there, we take $S(B) = \{ v_1, \dots, v_s \}$; we can then find a unitary group $G$ over $F$ associated to a pair $(B, \dagger)$, where:
\begin{enumerate}
\item $B$ is a central division algebra over $E$ of degree 4 which is split outside $S(B)$. For each place $w$ of $E$ lying above a place of $S(B)$, $B_w$ is a division algebra. 
\item $\dagger$ is an involution of $B$ of the second kind. 
\item Writing the functor of points of $G$ as $G(R) = \{ g \in (B \otimes_F R)^\times \mid g g^{\dagger \otimes 1} = 1 \}$, we have that $G(F \otimes_\bbQ \bbR)$ is compact, and for each finite place $v\not\in S(B)$ of $F$, $G_{F_v}$ is quasi-split. 
\end{enumerate}
(The existence of such a $G$ uses that $s = \# S(B)$ is even.) We set $\widetilde{S(B)} = \{ w_1, \dots, w_s \}$. As on \cite[p. 96]{Clo08}, we can find an integral model $\underline{G}$ of $G$ over $\cO_F$ and for each place $w$ of $E$ split over $F$ an isomorphism $\iota_w : G(F_v) \cong \GL_n(F_w)$ which takes $\underline{G}(\cO_{F_v})$ to $\GL_n(\cO_{E_w})$. If $v \in S(B)$ lies below $w \in \widetilde{S(B)}$, then we write $\rho_v : G(F_v) \to \GL_n(\cO)$ for a smooth representation of $G(F_v)$ such that $\mathrm{JL}( (\rho_v \otimes_\cO \overline{\bbQ}_p) \circ \iota_w^{-1}) \cong \iota^{-1}\Pi_w^\vee$. Thus $\rho_v$ is not uniquely determined (because there may be more than one choice of integral lattice), but $\rho_v \otimes_\cO \overline{\bbQ}_p$ is uniquely determined up to isomorphism. 

If $U = \prod_v U_v \subset G(\bbA_F^\infty)$ is an open compact subgroup, and $A$ is an $\cO$-module, then we will write $S(U, A)$ for the set of functions 
\[ f : G(F) \backslash G(\bbA_F^\infty) \to (\otimes_{v \in S(B)} \rho_v) \otimes_\cO A \]
such that for all $g \in G(\bbA_F^\infty)$, $u \in U$, we have $f(gu) = u_{S(B)}^{-1} f(g)$ (where $u_{S(B)}$ denotes the projection of $u$ to the $S(B)$ components). (This is the space denoted $S_{a, \{ \rho_v \}, \{ \chi_v \}}(U, A)$ in \cite{Clo08}; since we take $R = \emptyset$ and $a = 0$, and leave $\rho_v$ fixed throughout the proof, we omit this extra data from the notation.) 

If $T$ is a finite set of finite places of $F$, containing $S$, such that $U_v = \iota_w^{-1} \GL_n(\cO_{F_w})$ for all finite places of $E$ split over $F$ and prime to $T$, then (following \cite[p. 104]{Clo08}) we write $\bbT^T(U)$ for the commutative $\cO$-subalgebra of $\End_\cO(S(U, \cO))$ generated by unramified Hecke operators at finite places of $E$ split over $F$ and prime to $T$. 

We now pin down a particular choice of $U$ and $T$. By the Cebotarev density theorem, we can find a place $\widetilde{a}$ of $E$ prime to $p$, split over $F$ and absolutely unramified, such that $\ad \overline{\sigma}(\Frob_{\widetilde{a}}) = 1$ and $\ad \overline{\rho}(\Frob_{\widetilde{a}}) = \diag(1, i)$ in $\PGL_2(\bbF_5)$, where $i^2 = -1$. Moreover, we can assume that the residue characteristic of $\widetilde{a}$ is strictly greater than 3. Then $\widetilde{a}$ does not split in $E(\zeta_5)$, and consequently we have $q_{\widetilde{a}} \equiv -1 \text{ mod }5$. It follows that $h^2(G_{E_{\widetilde{a}}}, \ad \overline{\tau}) = h^0(G_{E_{\widetilde{a}}}, \ad \overline{\tau}(1)) = 0$, and the only deformations of $\overline{\tau}|_{G_{E_{\widetilde{a}}}}$ are the unramified ones. We take $a$ to be the place of $F$ below $\widetilde{a}$, and define $U = \prod_v U_v$, where:
\begin{itemize}
\item If $v$ is a place of $F$ inert in $E$, then $U_v$ is a hyperspecial maximal compact subgroup of $G(F_v)$.
\item If $v | p$, then $U_v = \underline{G}(\cO_{F_v})$.
\item If $v = v_1, \dots, v_s$, then $U_v$ is the unique maximal compact subgroup of $G(F_v)$ (i.e.\ the group of units in the ring of integers of the division algebra $B_{w}$).
\item If $v = a$, then $U_v = \iota_{\widetilde{a}}^{-1}( \ker \GL_n(\cO_{E_{\widetilde{a}}}) \to \GL_n(k({\widetilde{a}})) )$ (i.e.\ the inverse image under $\iota_{\widetilde{a}}$ of the principal congruence subgroup of $\GL_n(\cO_{E_{\widetilde{a}}})$). (Then $U_v$ is torsion-free, because $a$ is absolutely unramified and has residue characteristic $> 3$.)
\item If $v$ is any other place of $F$ split in $E$, then we take $U_v$ to be a maximal compact subgroup of $G(F_v)$.
\end{itemize}
Then the open compact subgroup $U \subset G(\bbA_F^\infty)$ is sufficiently small, in the sense of \cite[p. 98]{Clo08}, because of the choice of place $a$.
We set $T = S \cup \{ a \}$. Thus the Hecke algebra $\bbT^T(U)$ is now defined unambiguously.
The existence of the RACSDC automorphic representation $\Pi$ determines (by \cite[Proposition 3.3.2]{Clo08}) a homomorphism $\bbT^T(U) \to \overline{\bbQ}_p$ such that the unramified Hecke operator $\iota_w^{-1} T_w^i$ is mapped to the eigenvalue of $T_w^i$ on the 1-dimensional $\overline{\bbQ}_p$-vector space $\iota^{-1} \Pi_w^{\GL_n(\cO_{F_w})}$. The image of this homomorphism $\bbT^T(U) \to \overline{\bbQ}_p$ is contained in $\overline{\bbZ}_p$, and we write $\ffrm \subset \bbT^T(U)$ for the maximal ideal obtained by reduction modulo $p$. After possibly enlarging the coefficient field $L$, we can assume that $\bbT^T(U)_\ffrm = k$. By construction, the ideal $\ffrm$ is non-Eisenstein, in the sense of \cite[Definition 3.4.3]{Clo08}, and hence there is a lift of $\overline{\tau}$ to a  homomorphism 
\[ r_\ffrm : G_F \to \cG_n(\bbT^T(U)_\ffrm) \]
satisfying the list of conditions in \cite[Proposition 3.4.4]{Clo08}, with the exception that condition 6 (i.e.\ that $r_\ffrm$ is `Fontaine--Laffaille' at the primes above $p$) is replaced by the condition that $r_\ffrm$ is crystalline of weight 0, i.e.\ of type $\cD_v$ for each $v \in S_p$. In particular, these conditions imply the existence of a canonical surjection $R_\cS^\text{univ} \to \bbT^T(U)_\ffrm$, realizing $r_\ffrm$ as the pushforward of the universal deformation of $\overline{\tau}$. 

We will show that the closed subspace
\[ \Supp_{R_\cS^\text{univ}} S(U, \cO)_\ffrm \subset \Spec R_\cS^\text{univ} \]
contains $\Spec R_{\cS'}^\text{univ}$. This will imply the theorem. Indeed, the ring $\bbT^T(U)_\ffrm$ is a finite $\cO$-module, and the assertion on supports implies that the underlying reduced quotient of $R_{\cS'}^\text{univ}$ is a quotient of $\bbT^T(U)_\ffrm$. It follows from the completed version of Nakayama's lemma that $R_{\cS'}^\text{univ}$ is then itself finite over $\cO$. If $\varphi : R_{\cS'}^\text{univ} \to \overline{\bbQ}_p$ is a homomorphism then for the same reason it necessarily factors through $\bbT(U, \cO)_\ffrm$, implying (by \cite[Proposition 3.3.2]{Clo08} again) the existence of a RACSDC automorphic representation $\Pi'$ of $\GL_4(\bbA_E)$ such that $\varphi \circ r_\ffrm|_{G_E} \cong r_\iota(\Pi')$.

By a standard application of the Taylor--Wiles argument, as in the proof of \cite[Theorem 6.8]{ThorneB}, it is enough to show the existence of Taylor--Wiles systems, i.e.\ (following the proof of \cite[Proposition 4.4]{ThorneB}) to prove the following claim:
\begin{claim}
Let $S_p \subset S$ denote the set of places dividing $p$, and let $H^1_{\cL^\perp, S_p}(F_S/F, \ad \overline{\tau}(1))$ be the dual Selmer group defined on \cite[p. 27]{Clo08}. Then for every $N \geq 1$ and for every cohomology class 
\[ [\phi] \in H^1_{\cL^\perp, S_p}(F_S/F, \ad \overline{\tau}(1)), \]
 there exists an element $\sigma_0 \in G_{E(\zeta_{p^N})}$ and an eigenvalue $\alpha_0 \in k$ of $\overline{\tau}(\sigma_0)$ such that $\tr e_{\sigma_0, \alpha} \phi(\sigma_0) \neq 0$, where $e_{\sigma_0, \alpha} \in \End_k(\overline{\tau})$ is the unique $\overline{\tau}(\sigma_0)$-equivariant projector onto the $\alpha_0$-generalized eigenspace of $\overline{\tau}(\sigma_0)$. 
\end{claim}
Before establishing the claim, we make the following observations:
\begin{itemize}
\item The projective image $\ad \overline{\tau}(G_E)$ has no quotient of order $p$. (It is a quotient of $\PGL_2(\bbF_5) \times \overline{\sigma}(G_E)$.)
\item Let $L/E(\zeta_{p^N})$ denote the extension cut out by $\ad \overline{\tau}$. Then the intersection $H^1(L/F, \ad \overline{\tau}(1)) \cap H^1_{\cL^\perp, S_p}(F, \ad \overline{\tau}(1))$ (taken inside $H^1(F, \ad \overline{\tau}(1))$) is trivial. (Let $w = w_1$. The group 
\[ H^1(E(\zeta_{p^N})/F, \ad \overline{\tau}(1)^{G_{E(\zeta_{p^N})}}) \]
 is trivial, so the restriction map \[ H^1(L/F, \ad \overline{\tau}(1)) \to H^1(L/E(\zeta_{p^N}), \ad \overline{\tau}(1)) \cong H^1(\PSL_2(\bbF_5), \ad \overline{\rho}) \cong k \]
is injective, implying that $H^1(L/F, \ad \overline{\tau}(1))$ has dimension at most 1. The restriction map 
\[ H^1(L/F, \ad \overline{\tau}(1)) \to H^1(G_{E_w}, \ad \overline{\tau}(1)) \]
is injective, since the image of $G_{E_w}$ in $\Gal(L/F)$ is a $p$-Sylow subgroup. If $\cL_v^\perp \subset H^1(G_{E_w}, \ad \overline{\tau}(1))$ denotes the orthogonal complement of the tangent space $\cD_v(k[\epsilon]) = \cL_v \subset H^1(G_{E_w}, \ad \overline{\tau})$, then $\cL_v^\perp$ does not contain the image of $H^1(L/F, \ad \overline{\tau}(1))$. Therefore the intersection of $H^1(L/F, \ad \overline{\tau}(1))$ with the group 
\[ H^1_{\cL^\perp, S_p}(F, \ad \overline{\tau}(1)) \subset \ker\left[ H^1(F_S/F, \ad \overline{\tau}(1)) \to H^1(G_{E_w}, \ad \overline{\tau}(1)) / \cL_v^\perp \right] \]
is trivial.)
\item For each simple $k[G_{E(\zeta_p)}]$-submodule $W \subset \ad \overline{\tau}(1)$, there exists $\sigma \in G_F$ and an eigenvalue $\alpha \in k$ of $\overline{\tau}(\sigma)$ such that $\tr e_{\sigma, \alpha} W \neq 0$. (This is true for the individual factors $\overline{\rho}$ and $\overline{\sigma}$, so follows for their tensor product by the argument of \cite[Lemma A.3.1]{Gee13}.)
\end{itemize}
To prove the claim, let us therefore fix $N \geq 1$ and a class $[\phi] \in H^1_{\cL^\perp, S_p}(F_S/F, \ad \overline{\tau}(1))$, and consider the short exact sequence
\[ \xymatrix@1{ 0 \ar[r] & H^1(L/F, \ad \overline{\tau}(1)) \ar[r] & H^1(F_S/F, \ad \overline{\tau}(1)) \ar[r]^-{\Res^F_L} & H^1(F_S/L, \ad \overline{\tau}(1))^{G_F}.} \]
The class $[\phi]$ lies in the middle group; by the second point above, its image $f = \Res^F_L [\phi]$ in $H^1(F_S/L, \ad \overline{\tau}(1))^{G_F}$ is non-zero. We can view $f$ as a $G_{E(\zeta_p)}$-equivariant homomorphism $G_L \to \ad \overline{\tau}$. Let $W$ be a simple $k[G_{E(\zeta_p)}]$-submodule of the $k$-span of the image of $f$. By the third point above, we can find an element $\sigma \in G_{E(\zeta_p)}$, together with an eigenvalue $\alpha \in k$ of $\overline{\tau}(\sigma)$ such that $\tr e_{\sigma, \alpha} W \neq 0$. If $\tr e_{\sigma, \alpha} \phi(\sigma) \neq 0$, then we're done: take $\sigma_0 = \sigma$ and $\alpha_0 = \alpha$.

Otherwise, we can assume $\tr e_{\sigma, \alpha} \phi(\sigma) = 0$, in which case we choose $\tau \in G_L$ such that $\tr e_{\sigma, \alpha} f(\tau) \neq 0$. There is an eigenvalue $\alpha_0$ of $\sigma_0 = \tau \sigma$ such that $e_{\sigma, \alpha} = e_{\sigma_0, \alpha_0}$, and we can calculate
\[ \tr e_{\sigma_0, \alpha_0} \phi(\sigma_0) = \tr e_{\sigma, \alpha} (\phi(\tau) + \phi(\sigma)) = \tr e_{\sigma, \alpha} f(\tau) \neq 0. \]
This completes the proof.
\end{proof}

\subsection{Ordinary lifts under additional conditions}\label{sec_ordinary_lifts_under_additional_conditions}

We now prove a version of Theorem \ref{thm_existence_of_ordinary_lift_of_exceptional_representation} under supplementary hypotheses. Let $F$ be a totally real field such that for each $p$-adic place $v$ of $F$, $\bbQ_p(\zeta_p) \subset F_v$, and let $\pi$ be a RAESDC automorphic representation of $\GL_2(\bbA_F)$ of trivial central character and satisfying the following conditions:
\begin{itemize}
\item The representation $\pi$ everywhere unramified and of weight 2. For each place $v|p$ of $F$, $\pi_v$ is not $\iota$-ordinary.
\item The residual representation $\overline{\rho} = \overline{r_\iota(\pi)}$ is exceptional. For each place $v|p$ of $F$, $\overline{\rho}|_{G_{F_v}}$ is trivial.
\end{itemize}
We let $\rho' = r_\iota(\pi)$. Then $\det \rho' = \epsilon^{-1}$ and for each place $v|p$ of $F$, $\rho'|_{G_{F_v}}$ is crystalline non-ordinary. We suppose given as well the following additional data:
\begin{itemize}
\item Places $v_1, \dots, v_s$ of $F$ such that $s$ is even and for each $i = 1, \dots, s$, $q_{v_i} \equiv 1 \text{ mod }p$ and $\overline{\rho}(\Frob_{v_i})$ is unipotent of order $p$.
\item An everywhere unramified imaginary quadratic extension $E/F$, disjoint from $F(\zeta_5)$, in which $v_1, \dots, v_s$ and the $p$-adic places of $F$ all split. We write $w_1, \dots, w_s$ for a choice of places above $v_1, \dots, v_s$, respectively. We suppose there exists an everywhere unramified finite order character $\alpha : G_E \to \cO^\times$ such that $\alpha \alpha^c = \omega^{-1}$, the Teichm\"uller lift of the cyclotomic character. (We note that $\omega = \omega^{-1}$, since by hypothesis $E(\zeta_5)/E$ is a quadratic extension, but we preserve this notation for psychological reasons.)
\item An imaginary quadratic extension $K/F$, disjoint from $E(\zeta_5)$, in which every place $v|p$ of $F$ splits and $v_1, \dots, v_s$ are inert, and a continuous character $\overline{\theta} : G_K \to k^\times$ which is unramified outside $v_1, \dots, v_s$, satisfies $\overline{\theta} \overline{\theta}^c = \epsilon^{-1}$, and such that the restrictions $\overline{\theta}|_{I_{K_{v_i}}}$ have order divisible by $t_i$, for some prime $t_i | q_{v_i} + 1$.
\end{itemize}
We set $\overline{\sigma} = \Ind_{G_K}^{G_F} \overline{\theta}$; then $\overline{\sigma}$ is absolutely irreducible, because the local restrictions $\overline{\sigma}|_{G_{E_{v_i}}}$ are absolutely irreducible, and $\det \overline{\sigma} = \epsilon^{-1}$. Now choose for each place $v$ of $F$ above $p$ a place $v'$ of $K$ above $v$. We finally suppose given as well the following additional data:
\begin{itemize}
\item A continuous character $\theta : G_K \to \cO^\times$ lifting $\overline{\theta}$ which is unramified outside $v_1, \dots, v_s$ and the $p$-adic places and crystalline at the $p$-adic places, satisfies $\theta \theta^c = \epsilon^{-2} \omega$, and such that $\mathrm{HT}_\tau(\theta) = \{ 0 \}$ for each embedding $\tau : K \hookrightarrow \overline{\bbQ}_p$ inducing a place $v'$ of $K$.
\item A continuous character $\theta' : G_K \to \cO^\times$ lifting $\overline{\theta}$ which is unramified outside $v_1, \dots, v_s$ and the $p$-adic places and crystalline at the $p$-adic places, satisfies $\theta' (\theta')^c = \epsilon^{-1}$, and such that $\mathrm{HT}_\tau(\theta') = \{ 0 \}$ or $\{ 1 \}$ for each place $v'$ of $K$; and for each place $v'$ of $K$, there is at least one such embedding $\tau$ with $\mathrm{HT}_\tau(\theta') = \{ 0 \}$, and at least one such embedding with $\mathrm{HT}_\tau(\theta') = \{ 1 \}$.
\end{itemize}
We set $\sigma = \Ind_{G_K}^{G_F} \theta$ and $\sigma' = \Ind_{G_K}^{G_F} \theta'$. Then $\sigma$ is crystalline and ordinary, satisfying $\det \sigma = \epsilon^{-2} \omega$ and $\mathrm{HT}_\tau( \sigma ) = \{ 0, 2 \}$ for each embedding $\tau : E \hookrightarrow \overline{\bbQ}_p$, while $\sigma'$ is crystalline of weight 0 with $\det \sigma' = \epsilon^{-1}$ and for each place $v|p$ of $F$, $\sigma'|_{G_{F_v}}$ is non-ordinary.

We introduce the following deformation problem relative to $F$ (notation now as in the main body of the paper):
\[ \cS_{\overline{\rho}} = (\overline{\rho}, \epsilon^{-1}\omega, S_p, \{ \cO \}_{v \in S_p}, \{ \cD_v \}_{v \in S_p}), \]
where for each $v \in S_p$, $\cD_v$ is the local deformation problem of ordinary crystalline representations of Hodge--Tate weights $\{0, 2 \}$. The local lifting ring of $\cD_v$ is then irreducible, with smooth generic fiber. Let $\overline{\tau} = (\overline{\rho} \otimes \overline{\sigma})|_{G_E} \otimes \alpha$. Then we have
\[ \overline{\tau}^c \cong (\overline{\rho} \otimes \overline{\sigma})|_{G_E} \otimes \omega^{-1} \alpha^{-1} \cong (\overline{\rho} \otimes \overline{\sigma})^\vee|_{G_E} \otimes \epsilon^{-3} \omega \omega^{-1} \alpha^{-1} \cong \overline{\tau}^\vee \epsilon^{-3}. \]
Thus $\overline{\tau}$ is conjugate self-dual, and one sees that it satisfies the hypotheses of \S \ref{sec_automorphy_in_4_dimensions}. We can therefore fix an extension of $\overline{\tau}$ to a homomorphism $\overline{\tau} : G_F \to \cG_4(k)$ with $\nu \circ \overline{\tau} = \epsilon^{-3}$. The deformation problem $\cS$ is then defined (see equation (\ref{eqn_deformation_problem_s})). In order to pin down the subproblem $\cS'$ as in equation (\ref{eqn_deformation_problem_sprime}), it is necessary to specify a RACSDC automorphic representation $\Pi$ of $\GL_4(\bbA_E)$ with $\overline{r_\iota(\Pi)} = \overline{\tau}$, and we take it to be the one with Galois representation $(\rho' \otimes \sigma)|_{G_E} \otimes \alpha$, which exists by automorphic induction. Then the deformation problem $\cS'$ is also defined.

Theorem \ref{thm_automorphy_lifting_in_dimension_4} implies that the ring $R_{\cS'}$ is a finite $\cO$-algebra. Let us write $\rho^u : G_F \to \GL_2(R_{\cS_{\overline{\rho}}})$ for a representative of the universal deformation valued in this ring, and let $\tau^u = (\rho^u \otimes \sigma')|_{G_E} \otimes \alpha$, a 4-dimensional representation with coefficients in $ R_{\cS_{\overline{\rho}}}$. We have $(\tau^u)^c \cong (\tau^u)^\vee \epsilon^{-3}$, and we can therefore find an extension of $\tau^u$ to a homomorphism $\tau^u : G_F \to \cG_4(R_{\cS_{\overline{\rho}}})$. By construction, the representation $\tau^u$ is of type $\cS'$, so is classified by a homomorphism $R_{\cS'}^\text{univ} \to R_{\cS_{\overline{\rho}}}$. Just as in \cite[\S 3.2]{Bar12}, one can show that this is a finite morphism of rings, and hence that $R_{\cS_{\overline{\rho}}}$ is a finite $\cO$-algebra. One also knows by the standard argument in Galois cohomology that $\dim R_{\cS_{\overline{\rho}}} \geq 1$, implying that $R_{\cS_{\overline{\rho}}}$ must have a $\overline{\bbQ}_p$-point, corresponding to a representation $\rho : G_F \to \GL_2(\cO)$ of type $\cS_{\overline{\rho}}$ (after possibly enlarging $\cO$).

In the same way, the representation $(\rho \otimes \sigma')|_{G_E} \otimes \alpha$ gives rise to a representation of type $\cS'$, so Theorem \ref{thm_automorphy_lifting_in_dimension_4} implies that $(\rho \otimes \sigma')|_{G_E} \otimes \alpha$ is automorphic. It then follows from \cite[Proposition 5.1.1]{Bar12} that $\rho|_{G_E}$ itself is automorphic, and so the same is true for $\rho$ itself. By Hida theory, we can pass to a soluble extension $F'/F$ and replace $\rho$ by a representation of weight 2 (i.e.\ with $\tau$-Hodge--Tate weights $\{0, 1 \}$ for every $\tau$) with the desired properties. This proves Theorem \ref{thm_existence_of_ordinary_lift_of_exceptional_representation} under the conditions of this section.
\subsection{Ordinary lifts in general}\label{sec_ordinary_lifts_in_general}
We now complete the proof of Theorem \ref{thm_existence_of_ordinary_lift_of_exceptional_representation}. After making a preliminary soluble base change, we can assume that we are in the following situation (cf. \cite[Lemma 6.1.1]{Bar12}): we are given a totally real field $F$ and a cuspidal automorphic representation $\pi$ of $\GL_2(\bbA_F)$, of weight 2 and trivial central character, satisfying the following additional conditions:
\begin{itemize}
\item The representation $\overline{\rho} = \overline{r_\iota(\pi)}$ is exceptional, and for each place $v | p$ of $F$, $\overline{r_\iota(\pi)}$ is trivial.
\item For each place $v | p$ of $F$, $\pi_v$ is not $\iota$-ordinary and $\bbQ_p(\zeta_p) \subset F_v$.
\end{itemize}
By the Cebotarev density theorem, we can find a place $u_0$ of $F$ such that $q_{u_0} \equiv 1 \text{ mod }p$ and $\overline{\rho}(\Frob_{u_0})$ is unipotent of order $p$. (Observe that the condition $q_{u_0} \equiv 1 \text{ mod }p$ is equivalent to asking that the place $u_0$ split in the unique quadratic extension of $F$ which is contained in the field cut out by $\ad \overline{\rho}$, which has Galois group $\PGL_2(\bbF_5)$, by hypothesis.) Let $K/F$ be an imaginary quadratic extension disjoint from $F(\zeta_5)$ in which each place $v|p$ of $F$ splits and $u_0$ is inert. By \cite[Lemma A.2.5]{BLGGT}, we can find a continuous character $\overline{\theta} : G_K \to k^\times$ such that the order of $\overline{\theta}_{I_{u_0}}$ is divisible by a prime $t | (q_{u_0} + 1)$ and $\overline{\theta} \overline{\theta}^c = \epsilon^{-1}$ and for each place $v | p$ of $K$, $\overline{\theta}|_{I_{K_v}}$ is trivial.

Choose for each place $v|p$ of $F$ a place $v'$ of $K$ above $v$. By another application of \cite[Lemma A.2.5]{BLGGT}, we can find characters $\theta, \theta: G_K \to \cO^\times$ lifting $\overline{\theta}$ and satisfying the following conditions:
\begin{itemize}
\item We have $\theta \theta^c = \epsilon^{-2} \omega$, $\theta$ and for each embedding $\tau : K \hookrightarrow \overline{\bbQ}_p$ inducing a place $v'$, $\mathrm{HT}_\tau(\theta) = \{ 0 \}$.
\item We have $\theta' (\theta')^c = \epsilon^{-1}$ and for each embedding $\tau : K \hookrightarrow \overline{\bbQ}_p$ inducing a place $v'$, $\mathrm{HT}_\tau(\theta') = \{ 0 \}$ or $\{ 1 \}$; and for each $v'$, there exists at least one such embedding with $\mathrm{HT}_\tau(\theta') = \{ 0 \}$  and at least one such embedding with $\mathrm{HT}_\tau(\theta') = 1$.
\end{itemize}
We choose as well an imaginary quadratic extension $E/F$, disjoint from $K(\zeta_5)$ over $F$, in which every prime above $p$ or $u_0$ splits, and a finite order character $\alpha : G_E \to \cO^\times$ such that $\alpha \alpha^c = \omega^{-1}$. We now observe that we can find a soluble extension $F'/F$ of even degree, linearly disjoint over $F$ from $K \cdot E(\zeta_5)$, and satisfying the following conditions:
\begin{itemize}
\item The place $u_0$ splits in $F'$. We write $v_1, \dots, v_s$ for the places of $F$ above $u_0$.
\item Let $K' = F' \cdot K$. Then the characters $\theta|_{G_K}$ and $\theta'|_{G_K}$ are unramified at places not dividing $p$ or $u_0$, and crystalline and the places dividing $p$.
\item Let $E' = F' \cdot E$. Then the extension $E'/F'$ is everywhere unramified and the character $\alpha' = \alpha|_{G_{E'}}$ is everywhere unramified.
\end{itemize}
The hypotheses of \S \ref{sec_ordinary_lifts_under_additional_conditions} now apply to $F'$ together with $\theta|_{G_{K'}}$ etc. This concludes the proof of Theorem \ref{thm_existence_of_ordinary_lift_of_exceptional_representation}.

\section{Deduction of the  main theorem}\label{sec_main_theorem}

We now use the results established so far to deduce the theorem stated in the introduction. We first prove the theorem under favorable local hypotheses. We then use the results of \S \ref{ordinary} and soluble base change (in the form of \cite[Lemma 7.1]{ThorneA}) to show that the general case can be reduced to this one.

\begin{theorem}\label{intermediate}
 Let $F$ be a totally real field, let $p$ be an odd prime, and let $\rho:G_F \ra \GL_2(\aQp)$ be a continuous representation. Suppose that the following conditions hold.
 \begin{enumerate}
 \item  $[F:\Q]$ is even.

 \item The representation $\rhobar|_{F(\zeta_p)}$ is irreducible.
 
 \item  The character $\psi=\epsilon\det \rho$ is everywhere unramified.
 
 \item The representation $\rho$ is almost everywhere unramified.
 
 \item For each place $v|p$, $\rho|_{G_{F_v}}$ is semistable and $\rhobar|_{G_{F_v}}$ is trivial. For each embedding $\tau:F \ra \aQp$, we have ${\rm HT}_\tau(\rho)=\{0,1\}$.
 
 \item If $v$ is a finite place of $F$ not dividing $p$ at which $\rho$ is ramified, then $q_v \equiv 1 \text{ mod }p$, $\mathrm{WD}(\rho|_{G_{F_v}})^{\text{F-ss}} \cong {\rm rec}_{F_v}^T({\rm St}_2(\chi_v))$, for some unramified character $\chi_v:F_v^\times \ra \aQp^\times$, and $\rhobar|_{G_{F_v}}$ is trivial. The number of such places is even.
 
 \item   There exists a cuspidal automorphic representation $\pi$ of $\GL_2(\A_F)$ of weight 2  and an isomorphism $\iota: \aQp \ra \C$ satisfying the following conditions:
 
 \begin{enumerate}
 
 \item There is an isomorphism of residual representations $\overline{r_\iota(\pi)}=\overline{\rho}$.
 \item If $v |p $ and $\rho|_{G_{F_v}}$ is ordinary, then $\pi_v$ is $\iota$-ordinary and $\pi_v^{U_0(v)} \neq 0$. 
  If $v |p $ and $\rho|_{G_{F_v}}$ is non-ordinary, then $\pi_v$ is  not $\iota$-ordinary and $\pi_v$ is unramified. 
 \item If $v$ is place of $F$ that  does not divide  $ p\infty$  and $\rho|_{G_{F_v}}$ is unramified, then $\pi_v$ is unramified; if  $\rho_{G_{F_v}}$ is  ramified, then $\pi_v$ is an unramified twist of the Steinberg representation. 

 \end{enumerate}

\end{enumerate}

Then $\rho$ is automorphic: there exists a cuspidal automorphic representation $\pi'$ of $\GL_2(\A_F)$ of weight 2 and an isomorphism $\rho\simeq r_\iota(\pi')$.

\end{theorem}
\begin{proof}
The proof is essentially the same as the proof of \cite[Theorem 7.2]{ThorneA}. At this point enough minor deviations have accumulated that we give the details. By known results, we can assume that $p = 5$ and that $\overline{\rho}$ is exceptional. Choosing a coefficient field and conjugating, we can assume that $\rho$ takes values in $\GL_2(\cO)$ and that the residue field $k$ contains the eigenvalues of all elements in the image of $\overline{\rho}$. We will show that $\rho$ satisfies the hypotheses of Corollary \ref{approximation}, using Theorem \ref{cor_r_equals_t}. This will imply the theorem. Let us therefore fix an integer $N \geq 1$, let $\sigma \subset S_p$ denote the set of $p$-adic places where $\rho$ is ordinary, and let $R$ denote the set of places prime to $p$ where $\rho$ is ramified. We consider the global deformation problem:
\[ \cS=(\rhobar,\epsilon^{-1}\psi,S_p \cup R, \{\Lambda_v\}_{v \in \sigma} \cup \{\Oc\}_{v \in (S_p-\sigma) \cup R}, \{\cD_v^{\text{ord}}\}_{v \in \sigma} \cup \{\cD_v^{\text{non-ord}}\}_{v\in S_p -\sigma} 
\cup \{\cD_v^{\rm St}\}_{v \in R}),\]
and set $T=S_p \cup R$. Here the rings $\Lambda_v$ $(v\in \sigma)$ are the local Iwasawa algebras, as in \S \ref{sec_r_equals_t}. We now claim that we can find a finite set $Q_0$ of finite places of $F$ satisfying the following conditions:
\begin{enumerate}
\item We have $Q_0=\{v_1, v_2 \}$ where the deformation problems $\cD_{v_1}^{\rm St(\alpha_{v_1})}$ and $\cD_{v_2}^{\rm St(uni)}$ is defined. In particular, $\# Q_0$ is even.
\item The element $\rho_{N}(\Frob_{v_1})$ is conjugate to $\rho_{N}(c)$ for $c \in G_F$ a complex conjugation, and $q_{v_1} \equiv -1 \text{ mod }p^N$.
\item We have $q_{v_2} \equiv 1 \text{ mod }p$, but $q_{v_2} \not\equiv 1 \text{ mod }p^2$, $\rhobar(\Frob_{v_2})$ has order $p$, and $\rho_{N}(\Frob_{v_2})$ is of the form 
\[ \left(\begin{array}{cc}
  \chi  & *  \\
  0 & \chi\epsilon^{-1}
\end{array} \right),\]
 where $\chi$ is an unramified finite order character.
\end{enumerate}
Indeed, we choose $v_2$ to be any place satisfying the conditions of Proposition \ref{auxiliary}, and then choose $v_1$ to be any place such that $q_{v_1} \equiv -1 \text{ mod }p^{N}$ and such that $\rho_{N}(\Frob_{v_1})$ is conjugate to $\rho_{N}(c)$. Such a choice of $Q_0$ having been fixed, we define the augmented deformation problem
\[ \begin{split} \cS_{Q_0}=(\rhobar,\epsilon^{-1}\psi,S_p \cup R \cup Q_0, \{\Lambda_v\}_{v \in \sigma} \cup \{\Oc\}_{v \in (S_p-\sigma) \cup R}, \{\cD_v^{\text{ord}}\}_{v \in \sigma} \cup \{\cD_v^{\text{non-ord}}\}_{v\in S_p -\sigma} 
\cup \{\cD_v^{\rm St}\}_{v \in R} \\ \cup \{\cD_{v_1}^{\rm St(\alpha_{v_1})}, \cD_{v_2}^{\rm St(uni)}\}). \end{split} \]
By \cite[Lemma 4.13]{ThorneA}, we can find a cuspidal automorphic representation $\pi_0$ of $\GL_2(\A_F)$ weight 2 satisfying the following conditions:
\begin{itemize} 
\item There is an isomorphism of residual representations $\overline{r_\iota(\pi_0)} \cong \rhobar$.
\item If $v \in \sigma$, then $\pi_{0,v}$ is $\iota$-ordinary and $\pi_{0,v}^{U_0(v)} \neq 0$. If $ v \in S_p - \sigma$, then $\pi_{0,v}$ is not $\iota$-ordinary.
\item If  $v \notin S_p \cup R \cup Q_0$ is a finite place of $F$, then $\pi_{0,v}$ is unramified. If $v \in R \cup Q_0$, then $\pi_{0,v}$ is an unramified twist of the Steinberg representation. If $v = v_1$, then the eigenvalue of $U_v$ on $\iota^{-1}\pi_{0,v}^{U_0(v)} $ is congruent to $\alpha_v$ modulo the maximal ideal of $\aZp$. 
\end{itemize}
After replacing $\pi_0$ by a character twist, we can assume that $\pi_0$ has central character $\iota \psi$. The hypotheses  of \S \ref{sec_r_equals_t} are thus satisfied with respect to the global deformation problem $\cS_{Q_0}$ and the representation $\pi_0$. 

Fix a choice of auxiliary prime $a$ as in Lemma \ref{trivial}, let $S = S_p \cup R \cup \{ a \}$, and let $\ffrm_\emptyset \subset \bbT^{S, \text{univ}}$ be the maximal ideal corresponding to the representation $\pi_0$. Then $\ffrm_\emptyset$ is in the support of $H_R(U)$, where $U = U_1(\sigma)$ is as described in \S \ref{sec_r_equals_t}. Let $C_0 = \dim_k (H_R(U) \otimes_\cO k)[\ffrm_{\emptyset}]$. It follows from Proposition \ref{bound} that $\dim_k (H_{R \cup Q_0}(U_{Q_0} \otimes_\cO k)[\ffrm_{Q_0}] \leq 4^{\# Q_0} C_0$. We can therefore apply Theorem \ref{cor_r_equals_t} with $C = 4^{\# Q_0} C_0$ and $n =  1$ to conclude that there is a homomorphism $f : \bbT_{Q_0}^{S \cup Q_0}(H_{R \cup Q_0}(U_{Q_0})) \to \cO / \lambda^{\lfloor N / C \rfloor}$ with the following properties:
\begin{itemize}
\item For each finite place $v \not\in S \cup Q_0$ of $F$, we have $f(T_v) = \tr \rho_N(\Frob_v)$.
\item Let $I = \ker f$. Then  $(H_{R \cup Q_0}(U_{Q_0})_{\ffrm_{Q_0}} \otimes_\cO \cO/\lambda^{\lfloor N / C \rfloor})[I]$ contains an $\cO$-submodule isomorphic to $\cO / \lambda^{\lfloor N / C \rfloor}$.
\end{itemize}
Let $m$ be such that $q_{v_2} \not \equiv 1 \text{ mod }\lambda^{m+1}$. Since $C$ is independent of the choice of $N$, the conditions of Corollary \ref{approximation} are now satisfied with $r = 1$, allowing us to conclude the automorphy of $\rho$.
\end{proof}

\begin{lemma}\label{lem_existence_of_ordinary_lifts}
Let $F$ be a totally real field, let $p = 5$, and let $\iota : \overline{\bbQ}_p \cong \bbC$ be an isomorphism. Let $\pi$ be a cuspidal automorphic representation of $\GL_2(\bbA_F)$ of weight 2 such that $\overline{r_\iota(\pi)}$ is exceptional. Let $\sigma \subset S_p$ be a subset, and let $R$ be a finite set of finite places of $F$ which are prime to $p$. Then there exists a soluble totally real extension $F' / F$ and a cuspidal automorphic representation $\pi_0$ of $\GL_2(\bbA_{F'})$ of weight 2, satisfying the following conditions:
\begin{enumerate}
\item There is an isomorphism $\overline{r_\iota(\pi)}|_{G_{F'}} \cong \overline{r_\iota(\pi_0)}$. If $v$ is a place of $F'$ above $S_p \cup R$, then $\overline{r_\iota(\pi)}|_{G_{F'_v}}$ is trivial.
\item If $v$ is a $p$-adic place of $F'$ dividing $\sigma$, then $\pi_{0, v}^{U_0(v)} \neq 0$, and $\pi_v$ is $\iota$-ordinary.
\item If $v$ is a $p$-adic place of $F'$ not dividing $\sigma$, then $\pi_{0, v}$ is unramified and not $\iota$-ordinary.
\item If $v$ is a place of $F'$ dividing $R$, then $\pi_{0, v}$ is an unramified twist of the Steinberg representation and $q_v \equiv 1 \text{ mod }p$.
\item If $v$ is any other finite place of $F'$, then $\pi_{0, v}$ is unramified.
\end{enumerate}
\end{lemma}
\begin{proof}
We recall (Lemma \ref{bc}) that the property of being exceptional is preserved under soluble base change. The current lemma is now an easy consequence of Theorem \ref{thm_existence_of_ordinary_lift_of_exceptional_representation}, Lemma \ref{types}, and \cite[Lemma 4.13]{ThorneA}.
\end{proof}

We can now prove our main theorem.
\begin{theorem}\label{thm_main_theorem}
Let $F$ be a totally real number field, let $p$ be an odd prime, and let $\rho : G_F \to \GL_2(\barqp)$ be a continuous representation satisfying the following conditions.
\begin{enumerate}
\item The representation $\rho$ is almost everywhere unramified. 
\item For each place $v | p$ of $F$, $\rho|_{G_{F_v}}$ is de Rham. For each embedding $\tau : F \hookrightarrow \barqp$, we have $\mathrm{HT}_\tau(\rho) = \{0, 1\}$.
\item The  representation $\overline{\rho}|_{G_{F(\zeta_p)}}$ is irreducible.
\item There exists a cuspidal automorphic representation $\pi_0$ of $\GL_2(\bbA_F)$ of weight 2 such that $\overline{r_\iota(\pi)} \cong \overline{\rho}$.
\end{enumerate}
Then $\rho$ is automorphic: there exists a cuspidal automorphic representation $\pi$ of $\GL_2(\bbA_F)$ of weight 2, an isomorphism $\iota : \barqp \to \bbC$, and an isomorphism $\rho \cong r_\iota(\pi)$.
\end{theorem}
\begin{proof}
By known results, it suffices to treat the case where $p = 5$ and $\overline{\rho}$ is exceptional. Replacing $\rho$ by a twist, we can assume that $\epsilon \det \rho$ has order prime to $p$. By soluble base change and descent (\cite[Lemma 7.1]{ThorneA}), we are free to replace $F$ be any soluble totally real extension. In particular, after applying Lemma \ref{lem_existence_of_ordinary_lifts} and making a preliminary soluble base change we can assume that the following conditions are satisfied:
\begin{itemize}
\item For each finite place $v$ of $F$, $\rho$ is semi-stable. If $\rho|_{G_{F_v}}$ is ramified, then $\overline{\rho}|_{G_{F_v}}$ is trivial. 
\item The cuspidal automorphic representation $\pi_0$ satisfies the following conditions:
\begin{itemize}
\item If $v \in S_p$ and $\rho|_{G_{F_v}}$ is $\iota$-ordinary, then $\pi_{0, v}$ is $\iota$-ordinary and $\pi_{0, v}^{U_0(v)} \neq 0$. If $\rho|_{G_{F_v}}$ is not $\iota$-ordinary, then $\pi_{0, v}$ is unramified and not $\iota$-ordinary.
\item If $v$ is a finite place of $F$ prime to $p$, then $\pi_{0, v}$ is ramified if and only if $\rho|_{G_{F_v}}$ is ramified, and in this case $\pi_{0, v}$ is an unramified twist of the Steinberg representation and $q_v \equiv 1 \text{ mod }p$. The number of such places is even.
\end{itemize}
\end{itemize}
The hypotheses of Theorem \ref{intermediate} are now satisfied; this completes the proof.
\end{proof}

\subsection{Acknowledgements} CK  was supported by NSF grant DMS-1161671 and by a Humboldt Research Award, and thanks the Tata Institute of Fundamental Research, Mumbai for hospitality during the period in which  much  of the work on this paper was done.  He thanks G. B\"ockle,  F. Diamond,  N. Fakhruddin,  M. Larsen, D. Prasad, R. Ramakrishna and J-P. Serre for helpful conversations and correspondence. This research was partially conducted during the period JT served as a Clay Research Fellow.

\bibliographystyle{alpha}
\bibliography{gl2}

\end{document}